\documentclass[11pt,twoside]{amsart}
\usepackage{url}
\usepackage{amssymb,amsthm,amsfonts,amstext,amsmath}
\usepackage{newcent}       
\usepackage{helvet}         
\usepackage{courier}        
\usepackage{graphicx}
\usepackage{enumerate}\usepackage{mathrsfs}
\usepackage[charter]{mathdesign}
\usepackage{XCharter}
\usepackage[colorlinks=true,linkcolor=blue,citecolor=magenta]{hyperref}
\usepackage{bbm}
\usepackage{a4wide}
\usepackage{geometry}
\usepackage{multicol}
\numberwithin{equation}{section}
\usepackage{color}
\usepackage{pstricks}
\usepackage{float}
\usepackage{mathrsfs}
\usepackage{tikz}
\usepackage{comment}

\newtheorem{theorem}{Theorem}[section]
\newtheorem{lemma}[theorem]{Lemma}

\newtheorem{remark}[theorem]{Remark}
\newtheorem{definition}{Definition}
\newcommand{\mc}[1]{{\mathcal #1}}
\newcommand{\mf}[1]{{\mathfrak #1}}

\newcommand{\bb}[1]{{\mathbb #1}}

\newcommand{\<}{\langle}
\renewcommand{\>}{\rangle}

\newcommand{\mcb}[1]{{\mathcalboondox #1}}

\usepackage{tikz}
\usepackage{tikz-3dplot}
\usetikzlibrary{calc, shapes, backgrounds,arrows.meta,patterns}
\usetikzlibrary{decorations.markings,positioning}

\def\centerarc[#1](#2)(#3:#4:#5){\draw[#1] ($(#2)+({#5*cos(#3)},{#5*sin(#3)})$) arc (#3:#4:#5);}
\DeclareFontFamily{U}{BOONDOX-calo}{\skewchar\font=45 }
\DeclareFontShape{U}{BOONDOX-calo}{m}{n}{
  <-> s*[1.05] BOONDOX-r-calo}{}
\DeclareFontShape{U}{BOONDOX-calo}{b}{n}{
  <-> s*[1.05] BOONDOX-b-calo}{}
\DeclareMathAlphabet{\mathcalboondox}{U}{BOONDOX-calo}{m}{n}
\SetMathAlphabet{\mathcalboondox}{bold}{U}{BOONDOX-calo}{b}{n}
\DeclareMathAlphabet{\mathbcalboondox}{U}{BOONDOX-calo}{b}{n}

\let\oldtocsection=\tocsection
\let\oldtocsubsection=\tocsubsection
\let\oldtocsubsubsection=\tocsubsubsection
\renewcommand{\tocsection}[2]{\hspace{0em}\oldtocsection{#1}{#2}}
\renewcommand{\tocsubsection}[2]{\hspace{1em}\oldtocsubsection{#1}{#2}}
\renewcommand{\tocsubsubsection}[2]{\hspace{2em}\oldtocsubsubsection{#1}{#2}}
\DeclareRobustCommand{\SkipTocEntry}[5]{}

\definecolor{deeppink}{rgb}{1.0, 0.08, 0.58}

\usepackage{hyperref}
\hypersetup{
	colorlinks=true,
	linkcolor=blue,
	filecolor=magenta,      
	urlcolor=magenta,
}

\usepackage[most]{tcolorbox}

\newtcolorbox{Box2}[2][]{
	lower separated=false,
	colback=white,
	colframe=black,fonttitle=\bfseries,
	colbacktitle=black,
	coltitle=white,
	enhanced,
	attach boxed title to top left={yshift=-0.075in,xshift=0.15in},
	boxed title style={boxrule=0pt,colframe=white,},
	title=#2,#1}

\begin{document}
	
\title[Scaling limits for Rudvalis card shuffles]{Scaling limits for Rudvalis card shuffles}

\author{P. Gon\c calves}
\address[P. Gon\c calves]{Center for Mathematical Analysis, Geometry and Dynamical Systems, Instituto Superior T\'{e}cnico, Universidade de Lisboa, 1049-001 Lisboa, Portugal}
\email{pgoncalves@tecnico.ulisboa.pt}
\urladdr{\url{https://patriciamath.wixsite.com/patricia}}

\author{M. Jara}
\address[M. Jara]{Instituto de Matem\'atica Pura e Aplicada, Estrada Dona Castorina 110, 22460-320
	RIO DE JANEIRO, BRAZIL}
\email{mjara@impa.br}
\urladdr{\url{http://w3.impa.br/~monets/index.html}}

\author{R. Marinho}
\address[R. Marinho]{Universidade Federal de Santa Maria, Campus Cachoeira do Sul, Rod. Taufik Germano, 3013, 96503-205, Cachoeira do Sul, Brasil}
\email{rodrigo.marinho@ufsm.br}
\urladdr{\url{https://marinhor.weebly.com}}

\author{D. Moreira}
\address[D. Moreira]{Center for Mathematical Analysis, Geometry and Dynamical Systems, Instituto Superior T\'{e}cnico, Universidade de Lisboa, 1049-001 Lisboa, Portugal}
\email{david.tadeu.98@tecnico.ulisboa.pt}

\begin{abstract}
We consider the Rudvalis card shuffle and some of its variations that were introduced by Diaconis and Saloff-Coste in \cite{symmetrized}, and we project them to some stochastic interacting particle system. For the latter, we derive the hydrodynamic limits  and we study the equilibrium fluctuations. Our results show that, for these shuffles, when we consider an asymmetric variation of the Rudvalis shuffle,  the hydrodynamic limit in Eulerian scale is given in terms of a transport equation with a constant that depends on the exchange rates of the system; while for  symmetric and weakly asymmetric variations of this shuffle, in diffusive time scale,  the evolution is given by the solution of a martingale problem.
\end{abstract}


\maketitle


\section{Introduction}\label{introduction}

Fluid dynamics  are commonly modelled by stochastic interacting particle systems. With these models it may be possible to derive partial differential equations as the  limits, in some sense, of some observables of the system that evolve both in space and  time. Those are called \textit{scaling limits}~\cite{claudios}.

One of the most famous stochastic interacting  particle systems is the symmetric simple exclusion process (SSEP), where particles perform symmetric nearest-neighbor random walks on a graph with the exclusion rule, which forbids more than one particle at the same vertex. When the graph is the finite cycle $\bb Z_n=\bb Z/n \bb Z$, for instance, the total mass (number of particles) is conserved and one can recover the heat equation on the torus $\bb T=[0,1)$ as the limit, in probability, of the \textit{empirical measure}, which is a measure on $\bb T$ obtained by re-scaling time  diffusively  as $n^2$, re-scaling space by $n^{-1}$, and assigning mass $n^{-1}$ to each particle. The limit of the empirical measure is known as the \textit{hydrodynamic limit} and it can be seen as a Law of Large Numbers for the density of particles. A more challenging task is to prove a central limit theorem from which one can observe the equilibrium fluctuations of the process.

An interesting fact is that the SSEP can be obtained as a projection of a very famous card shuffle known as the \textit{interchange process}. In order to see this projection, one just has to color some cards in black, the remaining ones in red, and to associate black cards with particles and red cards with empty sites. This projection and its scaling limits were used in \cite{lacoin} and in \cite{wilson} to study the mixing times of both processes, i.e.~ the time at which the total variation distance between the law of the process and the stationary measure is smaller than some threshold $\varepsilon>0$, both in the interval and in the cycle. Inspired by this, in this work we  color the cards of a deck with black and red colours, but we study the corresponding stochastic interacting particle system obtained when we perform a peculiar card shuffle, the \textit{Rudvalis shuffle}. This card shuffle was introduced by Arunas Rudvalis in \cite{diaconis} and it can be explained as follows: remove the card on the top of a deck and insert it back, uniformly at random, either  at the bottom or at the second position from the bottom (see Figure~\ref{fig}).
\begin{figure}[H]
	\begin{center}
		\includegraphics[scale=0.10]{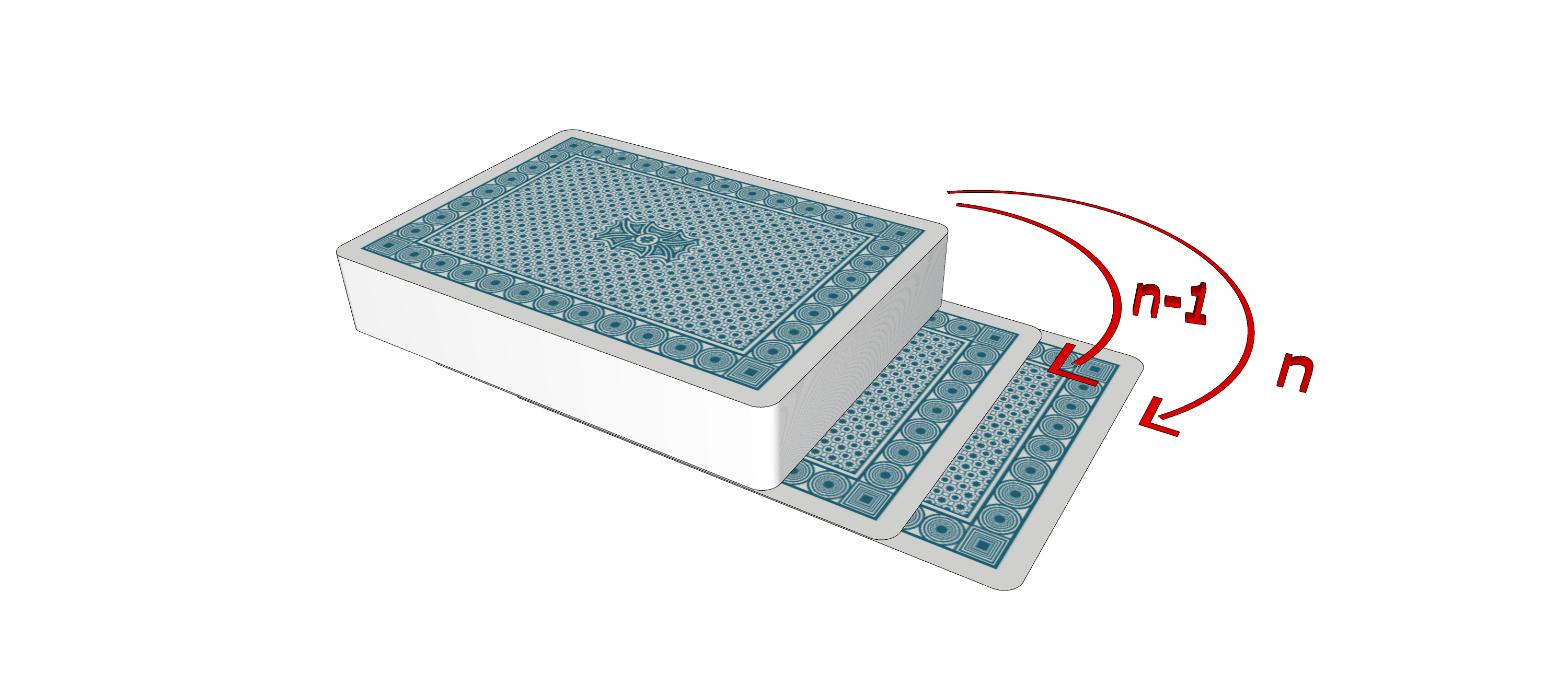}
		\caption{Insertion of the top card at position $n-1$ or $n$.}
		\label{fig}
	\end{center}
\end{figure}
Lower and upper bounds on the mixing time for the Rudvalis card shuffle were obtained in~\cite{wilson2} and in~\cite{hildebrand}. Matching bounds are  still unknown. In the case of the interchange process, the knowledge about the hydrodynamic limit and the fluctuations was crucial to determine the lower and upper bounds of the mixing time. Motivated by this, in this work we investigate the hydrodynamic behavior and the equilibrium fluctuations of the Rudvalis card shuffle and some of its variations. In~\cite{symmetrized} for instance, Diaconis and Saloff-Coste introduced a symmetric version of the Rudvalis shuffle where one inserts the top card at the bottom of the deck  with probability $1/4$, or inserts the bottom card at the top  of the deck with probability $1/4$,  or transposes the cards at positions $1$ and $2$ with probability $1/4$,  or nothing is done with probability $1/4$. In order to obtain distinct hydrodynamic behaviours for the associated stochastic interacting particle systems,  we introduce a continuous-time random walk on the symmetric group from which not only  we can get the symmetric and asymmetric variations of the Rudvalis card shuffle, but also a weakly asymmetric one. We prove that, in the Eulerian  time scale, the hydrodynamics and the equilibrium fluctuations of the projection of the asymmetric Rudvalis shuffle are driven by a transport equation with periodic boundary conditions. Furthermore, we show that, in the diffusive time scale, hydrodynamics and fluctuations for the symmetric and weakly asymmetric variations are driven by a solution of a non-trivial martingale problem. In~\cite{kbottom}, Goel introduced a variation of the Rudvalis card shuffle where one puts the top card uniformly at random in the bottom-$k$ positions of the deck. For this reason, this card shuffle is known as the Top to bottom-$k$ shuffle. For the top-to-bottom $k$-shuffle, our methodology yields the  hydrodynamic limits when $k=o(n)$ and the equilibrium  fluctuations when $k=o(\sqrt{n})$.

This work is organized as follows: in Section~\ref{s2} we introduce the continuous-time card shuffle and we map it to a stochastic  interacting particle systems. Moreover, in that section we state our main results, and to do so, we introduce some notation and definitions as Sobolev spaces, the empirical measure, the density fluctuation field, some weak solutions of partial differential equations and some martingale problems. In Section \ref{s3}  we prove the hydrodynamic limits in the Eulerian and diffusive time scales,  and in Section \ref{s5}  we obtain the scaling limit of  the equilibrium fluctuations in both time scales. 
In Section \ref{sec5}, we prove the uniqueness of the solution of the martingale problem obtained in diffusive time scale. 

\section{Notation and results\label{s2}}

\subsection{The Generalized Rudvalis shuffle}
Let $n\in\mathbb N$ be a scaling parameter and fix four sequences of real numbers $\{a_n\}_{n},\{b_n\}_n,\{c_n\}_n,\{d_n\}_n  \in [0,1]$ converging, as $n \to \infty$, respectively, to $a,b,c,d $.  
Let $S_n$ denote the symmetric group. $S_n$ must be understood as the set of all possible permutations of a deck with $n$ cards. Given a permutation $\sigma \in S_n$, we denote the card at position $x$ by $\sigma(x)$. Given $\sigma\in S_n$ and $x<y$, if we denote $\sigma=(\sigma(1),\sigma(2),\cdots,\sigma(x),\cdots, \sigma(y),\cdots, \sigma(n))\in S_n$, we define $\sigma^{x \rightarrow y}$ as the permutation obtained from $\sigma$ after we remove the card at position $x$ from the deck and insert it at position $y$, i.e.~  $$\sigma^{x \rightarrow y}=(\sigma(1),\sigma(2),\cdots, \sigma(x-1),\sigma(x+1),\cdots, \sigma(y),\sigma(x),\sigma(y+1),\cdots, \sigma(n)),$$
while for $x>y$ and $\sigma=(\sigma(1),\sigma(2),\cdots,\sigma(y),\cdots, \sigma(x),\cdots, \sigma(n))$ we define 
$$\sigma^{x \rightarrow y}=(\sigma(1),\sigma(2),\cdots, \sigma(y-1),\sigma(y),\sigma(x),\sigma(y+1),\cdots, \sigma(x-1),\sigma(y+1),\cdots, \sigma(n)).$$
We also define  $\sigma^{x \leftrightarrow y}$ as  the permutation  obtained after we transpose the cards at positions $x$ and $y$, i.e.~
for $\sigma=(\sigma(1),\sigma(2),\cdots,\sigma(x-1),\sigma(x),\sigma(x+1),\cdots, \sigma(y-1),\sigma(y),\sigma(y+1),\cdots, \sigma(n))$, we define 
 $$\sigma^{x \leftrightarrow y}=\sigma=(\sigma(1),\sigma(2),\cdots,\sigma(x-1),\sigma(y),\sigma(x+1),\cdots, \sigma(y-1),\sigma(x),\sigma(y+1),\cdots, \sigma(n)).$$

 We  consider the continuous-time Markov chain $\{ \sigma_t \}_{t \geq 0 }$ with state space $S_n$ and generator $\mcb L_n^{\spadesuit} $ given on $f: S_n \to \bb R$ and $\sigma\in S_n$ by
\begin{flalign}\label{generator}
\mcb L_n^{\spadesuit} f(\sigma)= a_n \left( f(\sigma^{1 \rightarrow n-1})- f(\sigma)\right) + b_n \left( f(\sigma^{1 \rightarrow n})- f(\sigma)\right) + c_n \left( f(\sigma^{n \rightarrow 1})- f(\sigma)\right) +d_n \left( f(\sigma^{1 \leftrightarrow 2})- f(\sigma)\right)
\end{flalign}
 and we call $\{ \sigma_t\}_{ t \geq 0 }$ the \textit{generalized Rudvalis card shuffle}.

\begin{remark}\label{remark}
We observe that for $a_n=b_n=1/2$ and $c_n=d_n=0$ we obtain the original Rudvalis shuffle introduced in~\cite{diaconis}. Similarly, when $a_n=0$ and $b_n=c_n=d_n=1/4$ we obtain the symmetric version of the Rudvalis  shuffle introduced in~\cite{symmetrized}. In~\cite{wilson2} it is also considered the case $a_n \in (0,1)$, $b_n=1-a_n$ and $c_n=d_n=0$.
\end{remark}

Our aim is to project the generalized Rudvalis shuffle to a stochastic  interacting particle system, where the latter is defined as a projection of the colours of the cards. Let $\bb Z_n=\bb Z/ n \bb Z$ and let  us colour each card of the deck in either red or black. For each permutation $\sigma$ and each $x,y \in \bb Z_n$, we define the colouring  $\eta=\eta(\sigma)$ by
$$
\eta(x)=\begin{cases}
1 &, \text{ if the card } \sigma(x) \text{ is black},\\
0 &, \text{ if the card } \sigma(x) \text{ is red}.
\end{cases}
$$
Let $\Omega_n=\{0,1\}^{\bb Z_n}$. For each $x, y \in \bb Z_n$ define $\eta^{x \rightarrow y}=\eta(\sigma^{x \rightarrow y})$ and $\eta^{x \leftrightarrow y}=\eta(\sigma^{x \leftrightarrow y})$. The projection defined above allows mapping black cards to particles and red cards to holes. From  \eqref{generator}, we see that  $\{\eta_t \}_{t \geq 0 }$ has an infinitesimal generator given on $f: \Omega_n \to \bb R$ and $\eta\in\Omega_n$ by
\begin{flalign*}
\mcb L_n f(\eta)= a_n \left( f(\eta^{1 \rightarrow n-1})- f(\eta)\right) + b_n \left( f(\eta^{1 \rightarrow n})- f(\eta)\right) + c_n \left( f(\eta^{n \rightarrow 1})- f(\eta)\right) +d_n \left( f(\eta^{1 \leftrightarrow 2})- f(\eta)\right).
\end{flalign*}
We will consider the process speeded up in the time scale $tn^\beta$ with $\beta>0$ and the generator of $\{\eta_{tn^\beta}\}_{t\geq0}$ is simply given by $n^\beta\mcb L_n$. We will consider two cases, the hyperbolic time scale, corresponding to the choice $\beta=1$ and  the diffusive time scale, corresponding to the choice $\beta=2$.

\subsection{Stationary measures}
We assume that  $a_n,b_n,c_n,d_n$ are chosen in such a way that the  processes $\{\sigma_t\}_{  t \geq 0 }$ and $\{\eta_t\}_{  t \geq 0 }$ are irreducible and in those cases, the uniform measure on $S_n$ will be the stationary measure of the card shuffle $\{\sigma_t\}_{  t \geq 0 }$. Therefore,  if the initial number of particles $\ell$ is fixed, then the uniform measure on the hyperplane $\Omega_{n,\ell}:=\{ \eta \in \Omega_n ; \sum_{x \in \bb Z_n} \eta(x)=\ell \}$ is the stationary measure of the particle system.

For $\varrho\in [0,1]$, let $\nu_\varrho$ be the Bernoulli product measure with parameter $\varrho $, that is, for each $\eta \in \Omega_n$,
\begin{flalign}\label{bernoulli}
\nu_\varrho (\eta)=\prod_{x \in \bb Z_n} \left\{ \varrho \eta(x) +(1-\varrho)(1-\eta(x)) \right\}.
\end{flalign}
It is straightforward to check that $\nu_\varrho$ is stationary for any $\varrho \in [0,1]$.
In order to define the space where our functionals will act, we need to introduce the so-called Sobolev spaces. 

\subsection{Sobolev Spaces}

Let $\bb T$ be the unit torus.  For $f,g:\bb T \to \bb C$, let  $\left<f,g \right> $ denote the inner product in $L^2(\bb T)$, which is given by $$\left<f,g \right>=\int_{\bb T} f(u) \overline{g(u)} du.$$ For each $k \in \bb Z$ let $\psi_k:\bb T \to \bb C$ be defined by
\begin{flalign}\label{psi}
	\psi_k(u)=\begin{cases}
		\sqrt{2} \sin (2 \pi k u) &, \text{ if } k<0,\\
		1 &, \text{ if } k=0,\\
		\sqrt{2} \cos (2 \pi k u) &, \text{ if } k>0.
	\end{cases}
\end{flalign}
for any $u \in \bb T$. The family $\{ \psi_k \}_{ k \in \bb Z }$ forms an orthonormal basis of $L^2(\bb T)$. Note that the functions $\psi_k$ are eigenfunctions of the positive symmetric linear operator $\mcb L=(1-\Delta)$, i.e.~ for each $k \in \bb Z$,
$
\mcb L \psi_k=\gamma_k \psi_k,
$
with $\gamma_k=1+4 \pi^2 k^2$.
Now we are ready to define Sobolev spaces.  For each positive integer $m$, we denote by $\mcb H_{m}$ the Hilbert space obtained as the completion of $C^\infty(\bb T)$ endowed with the inner product $\left\< \cdot,\cdot\right\>_{m}$ defined by
$$
\left\< f,g\right\>_{m}=\left\< f,\mcb L^m g \right\>.
$$
It is easy to verify that $\mcb H_m$ is the subspace of $L^2(\bb T)$ consisting of all functions $f$ such that
$$
\sum_{k \in \bb Z} |\left\< f, \psi_k \right\>|^2 \gamma_k^m<\infty.
$$

For each positive integer $m$, we denote by $\mcb H_{-m}$ the dual of $\mcb H_m$ with respect to the inner product  $\langle\cdot, \cdot \rangle$. It  is easy to check  that $\mcb H_{-m}$ consists of all sequences $\{ \left\<f,\psi_k \right\>\}_{  k \in \bb Z }$ such that
$$
\sum_{k \in \bb Z} |\left\< f, \psi_k \right\>|^2 \gamma_k^{-m}<\infty
$$
and that the inner product $\left\< f,g\right\>_{-m}$ of two functions $f,g \in \mcb H_{-m}$ can be written as
\begin{equation}\label{eq:norm-}
\left\< f,g \right\>_{-m}= \sum_{k \in \bb Z} \left\< f, \psi_k \right>  \left\< g, \psi_k \right> \gamma_k^{-m}.
\end{equation}
If we denote $L^2(\bb T)$ by $\mcb H_0$, then
from the explicit characterization of $\mcb H_{-m}$, we obtain the following inclusion of spaces
$$
\cdots \subset \mcb H_2 \subset \mcb  H_1 \subset \mcb H_0 \subset \mcb H_{-1} \subset \mcb H_{-2} \subset \cdots
$$

\subsection{Empirical measure}

We note that the number of particles (or the numbers of black cards) is conserved  by the dynamics  of $\{\eta_t\}_{t\geq 0}$. Our first question is related to understanding how does  the distribution of particles evolve in space and time.   To properly answer this question we need to define the empirical measure associated to the conserved quantity, i.e.~ the number of particles. To that end, 
let $\bb T$ be the unit torus and  fix a time horizon $[0,T]$, with $T>0$. We denote by $\mcb  M$  the space of positive measures on $\bb T$ with total mass bounded by one, endowed with the weak topology.

For each $t \in [0,T]$, let $\pi_t^{n,\beta} \in \mcb M$ be the empirical measure  on $\bb T$ obtained by rescaling space by $n^{-1}$, and assigning mass $n^{-1}$ to each particle (black card), that is,
\begin{flalign*}
\pi_t^{n,\beta}(\eta,du):=\frac{1}{n}\sum_{x\in \bb Z_n}  \eta_{tn^\beta}(x)\delta_{\tfrac xn}\, (du),
\end{flalign*}
where $\delta_u$ is the Dirac measure concentrated on $u\in\mathbb T$. For convenience of notation,  we denote the integral of  $f:\bb T \to \mathbb R$ with respect to $\pi_t^{n,\beta}$ as
$
\langle \pi_t^{n,\beta},f\rangle$.
Let $\{\mu_n\}_{n\in\bb N}$ be a sequence of probability measures in $\Omega_n$. For  $n \in \bb N$, let $\mathbb{P}_{\mu_n}$ be the probability measure on the Skorokhod space $\mcb {D}([0,T],\Omega_n)$ induced by  $\{\eta_{tn^\beta}\}_{t \geq 0}$ and the initial measure $\mu_n$ and  let $\mathbb{E}_{\mu_n}$ be the expectation with respect to $\mathbb{P}_{\mu_n}$. Let $\{\bb {Q}_{n,\beta} \}_{n\in\bb N}$ be the sequence of probability measures on $\mcb{D} ([0,T],  \mcb {M})$ induced by $\{\pi_{t}^{n,\beta}\}_{t \geq 0}$ and $\mathbb{P}_{\mu_n}$.

\begin{definition}\label{associated profile}
Let $\mathfrak h: \mathbb{R}\rightarrow[0,1]$ be a measurable function and $\{\mu_n\}_{n\in\
bb N}$ a sequence of probability measures in $\Omega_n$. We say that $\{\mu_n\}_{n\in \bb N}$  is \textit{associated with $\mathfrak h$}, if for any continuous function $G:\mathbb T\rightarrow \mathbb R$ and any $\delta > 0$, 
\begin{equation*}
\lim _{n\to \infty } \mu_n \Big( \eta \in \Omega_n : \Big|  \frac{1}{n} \sum_{x\in \bb Z_n} G(\tfrac{x}{n}) \eta(x)   - \int_{\mathbb{R}} G(u)\mathfrak h(u)\,du \, \Big|    > \delta \Big)= 0.
\end{equation*} 
\end{definition}

\subsection{Hydrodynamic limit}

In order to state our first result, we introduce the notion of weak solution of a transport equation on $\bb T$:
\begin{definition}\label{def_transport}
Let $\mathfrak g :\bb T \to [0,1]$ be a measurable function. Let $\theta \in \bb R$. We say that $\varrho^\theta: [0,T] \times \bb T \to [0,1]$ is a weak solution of the transport equation
	\begin{flalign}\label{transp_equation}
	\begin{cases}
	\partial_t \varrho^\theta (t,u) = \theta \nabla \varrho^\theta(t,u) &,\forall u \in \bb T; \\
	\varrho(0,u)=\mathfrak g(u) &, \forall u \in \bb T.
	\end{cases}
	\end{flalign}
	if for all $f \in C^{1}(\bb T)$ and all $t \in[0,T]$ we have
$$
	0=\left\< \varrho^\theta_t  - \mathfrak g , f\right\>  +\theta \int_0^t \langle  \varrho_s^\theta ,\nabla f\rangle \,ds.
$$
\end{definition}
Equation \eqref{transp_equation} has a unique weak solution given by $\varrho_t^\theta(u)=\varrho_0^\theta(u+\theta t)$. 
Now we state the  \textit{hydrodynamic limit} in the hyperbolic time scale ($\beta=1)$:

\begin{theorem}\label{hydrodynamic limit 1}
Let $\beta=1$ and $\varrho_0: \bb T \to [0,1]$ be a measurable function.    Let $\{\mu_n\}_{n\ \geq 1}$ be a sequence of probability measures in $\Omega_n$ associated to the profile $\varrho_0$. 
Assume that \begin{equation}\label{def_kappa}
\lim_{n\to+\infty}|a_n+b_n-c_n|=\kappa:=a+b-c\neq 0.
\end{equation}
Then, for any $t \in [0,T]$, the sequence  $\{{\bb Q}_{n,1}\}_{ n \in \bb N}$ converges, as $n\to+\infty$, with respect to the weak topology of $\mcb D ([0,T],\mcb H_{-m})$, for $m>3/2$, to ${\bb Q}_{1}$, which is a Dirac measure on the trajectory $\pi^1_t(du)=\varrho^\kappa(t,u)du$, where $\varrho^\kappa$ is the unique weak solution of  (\ref{transp_equation}) with $\theta=\kappa$.
\end{theorem}

We observe that since the limiting equation is a transport equation, we could perform a Galilean transformation at the level of the empirical measure and, in the hyperbolic time scale, we would  have a trivial evolution (which also happens if $\kappa=0$). 	In that sense, we need to speed up the time scale in order to see a non-trivial limit. In fact, the trivial behavior goes up to $\beta<2$. Therefore, we consider now the diffusive time scale corresponding to $\beta=2$. In this regime the limit is not longer deterministic since the contribution of the martingale term is relevant and brings an additional term to the equation.  Let us now define properly the notion of solution to the equation that we obtain in this regime.

\begin{definition}\label{def:mart_problem}
Let $\nu,\sigma>0$ and $\mcb v\in\mathbb R$. Fix $m >0$.  We say that a process $\{X_t\}_{ t \geq 0}$ with continuous trajectories in $\mcb  H_{ -m}$ is a solution of the \emph{martingale problem} associated to
 	\begin{equation}
\label{arica}
\partial_t X_t = \nu \Delta X_t + \mcb v \nabla X_t + \sigma\nabla X_t d \mc B_t,
\end{equation}
where  $\{ \mc B_t\}_{t \geq 0}$ is  a standard Brownian motion,
 if for every $f \in  C^\infty(\bb T)$ the process $\{\mcb M_t(f)\}_{ t \geq 0}$ given by
\[
\mcb M_t(f) := X_t(f) - X_0(f) - \int_0^t X_s(  \nu \Delta f - \mcb v \nabla f) ds  
\]
for every $t \geq 0$, is a martingale with respect to the filtration $\mathcal F_t:=\sigma \{ X_s(g);\,\,s \leq t \text{ and } g \in C^\infty(\bb T)  \}$ whose quadratic variation is given by
\begin{equation}\label{QVofmp}
\< \mcb M(f) \>_t := \int_0^t  \sigma^2 X_s(\nabla f)^2 ds.
\end{equation}
\end{definition}
Since our proofs combine tightness and the unique characterization of the limit point, the next result is crucial in what follows. We did not find in the literature the proof of this result, therefore we present it in Section \ref{sec5}.
\begin{theorem}
\label{theo:uni_mp} 
Any two  solutions of the martingale problem  associated to \eqref{arica} starting from the same initial condition have the same law. 
\end{theorem}

Now we state our second result.

\begin{theorem}\label{hydrodynamic limit 3}
Let $\beta=2$ and $\varrho_0: \bb T \to [0,1]$ be a measurable function.    Let $\{\mu_n\}_{n\in\bb N}$ be a sequence of probability measures in $\Omega_n$ associated to the profile $\varrho_0$. 
	Assume that $c>0$ and that $$a_n+b_n-c_n=\gamma/n$$ for some constant $\gamma \in \bb R$.  Then, the sequence $\{{\bb Q}_{n,2}\}_{ n \in \bb N}$  converges, as $n\to+\infty$, with respect to  the weak topology of $\mcb  D ([0,T],\mcb  H_{-m})$, for $m>5/2$,  to $\bb Q_2$  which is concentrated on the solution of the martingale problem associated to \eqref{arica} for $\nu=c$, $\mcb v=\gamma$ and $\sigma=\sqrt {2c}$ and with initial condition $\varrho_0$.

\end{theorem}

Now that we have established the hydrodynamic limit, which is a Law of Large Numbers for the empirical measure, we can analyse the fluctuations around that limit. That is the purpose of the next section, though we consider the process starting from the invariant state $\nu_\varrho$, since the non-equilibrium scenario is much more complicated and the equilibrium case already gives interesting results. 

\subsection{Equilibrium density fluctuations}

In the last subsection we proved the hydrodynamic limit in two different time scales: in the hyperbolic time scale the limit is deterministic, but in the diffusive time scale, the limit is random. Now we analyse the 
corresponding central limit theorems starting from the stationary measure $\nu_\varrho$. To that end, fix $\varrho \in [0,1]$. We define the \textit{density fluctuation field} on functions $f\in \mcb H_{m}$, by
\begin{flalign*}
	\mcb Y_t^{n,\beta}(f)=\frac{1}{\sqrt{n}} \sum_{x \in \bb Z_n} \left( \eta_{tn^\beta}(x) - \varrho \right) f\left( \tfrac{x}{n} \right).
\end{flalign*}
Now, we denote by $\widetilde {\bb Q}_{n,\beta} $ the probability measures on the Skorohod space  $\mcb D ([0,T],\mcb H_{-m})$ induced by the probability measure $\bb P_{\nu_\varrho}$ and the fluctuation field $\{\mcb Y_t^{n,\beta}\}_{t\geq 0}$. Let also $\mcb C ([0,T],\mcb H_{-m})$  be the space of  continuous trajectories with values in $\mcb H_{-m}.$  Recall the value of $\kappa$ given in \eqref{def_kappa}.

\begin{theorem}\label{clt1}
Let $\beta=1$. Assume that $$\lim_{n\to+\infty}|a_n+b_n-c_n|=\kappa\neq 0$$ and that  we start the process from the Bernoulli product measure $\nu_\varrho$, with $\varrho \in [0,1]\,$.
For $m>3/2$, let $\widetilde{\bb Q}_{1}$ be the probability measure on 
$\mcb C ([0,T],\mcb H_{-m})$ corresponding to a stationary Gaussian process with mean $0$ and covariance
given on $f,g\in\mcb H_{m}$ by
\begin{equation}\label{cov_field}
\bb E[\mcb Y_t(f)\mcb Y_s(g)] =\varrho(1-\varrho)\langle T^\kappa_{t-s}\,f,g\rangle,
\end{equation} where $T_t f(u)=f(u+\kappa t)$ is the semigroup associated to the transport equation given in \eqref{transp_equation} with $\theta=\kappa$.
Then the sequence $\{\widetilde {\bb Q}_{n,1}\}_{n\in\bb N} $ converges weakly, as $n\to+\infty$, to    $\widetilde{\bb Q}_{1}$.
\end{theorem}

\begin{theorem}\label{clt 3}
Let $\beta=2$. Assume that $c>0$, that $$a_n+b_n-c_n=\gamma/n$$ for some constant $\gamma \in \bb R$ and  that the system  starts  from the Bernoulli product measure $\nu_\varrho$, with $\varrho \in [0,1]\,$.  Then, the sequence $\{\widetilde {\bb Q}_{n,2}\}_{ n \in \bb N}$  converges with respect to  the weak topology of $\mcb  D ([0,T],\mcb  H_{-m})$, as $n\to+\infty$, for $m>5/2$,  to $\widetilde{\bb Q}_2$  which is concentrated on the solution of the martingale problem associated to \eqref{arica} for $\nu=c$, $\mcb v=\gamma$ and $\sigma=\sqrt {2c}$ and with initial condition  $\mcb Y_0$,  such that for each $f \in C^\infty(\bb T)  $,  $\mcb Y_0(f)$ is a normal random variable with mean zero and variance $\varrho (1-\varrho)\|f\|_{L^2(\bb T)}^2$.
\end{theorem}

\section{Hydrodynamic limits}\label{s3}

In this section we prove Theorems \ref{hydrodynamic limit 1} and \ref{hydrodynamic limit 3}. Recall the definition of  ${\bb Q}_{n,\beta}$ given above Definition \ref{associated profile}. We first show that $\{{\bb Q}_{n,\beta}\} _{ n \in \bb N}$ is tight with respect to the weak topology of $\mcb D([0,T],\mcb H_{-m})$ for any $m>3/2$ when $\beta=1$ and $m>5/2$ when $\beta=2$. Then we characterize the limit points. 

\subsection{Tightness}

In order to prove tightness, we follow~\cite{HS} and we invoke the following criterion.

\begin{lemma}\cite[Chapter 11, Lemma 3.2]{claudios}\label{criterion}

	The sequence of probability measures $\{\bb Q_{n,\beta}\}_{ n \in \bb N}$ on $\mcb D ([0,T], \mcb H_{-m})$ is tight if
	\begin{enumerate}
		\item 
		$\lim_{A\to \infty} \limsup_{n \to \infty} {\bb Q}_{n,\beta} \left(  \sup_{t \in [0,T]} 
		\|\pi_t^{n,\beta}\|_{-m}>A \right)=0$;	\item $\lim_{\delta \to 0} \limsup_{n \to \infty} {\bb Q}_{n,\beta}\left( \sup_{\stackrel{|s-t|\leq \delta}{0\leq s, t \leq T}} \|\pi_t^{n,\beta} - \pi_s^{n,\beta}\|_{-m} \geq \varepsilon \right)=0$ for every $\varepsilon>0$.
	\end{enumerate}
\end{lemma}
In order to prove the validity of (1) and (2) above,  we use the martingale characterization of the empirical measure. We can recast (1) and (2) in terms of $\mathbb P_{\mu_n}$ instead  of  ${\bb Q}_{n,\beta} $. 
From  Dynkin's formula (see for instance~\cite[Appendix 1, Lemma 5.1]{claudios}), for  any $f \in C^{\infty}( \bb T)$,
\begin{flalign}\label{eq:martingale}
	&\mcb M_t^{n,\beta}(f):=\langle \pi_{t}^{n,\beta},f\rangle -\langle \pi_0^{n,\beta},f\rangle - \int_0^t ( \partial_s + n^{\beta} \mcb L_n ) \langle \pi_{s}^{n,\beta}, f\rangle ds
\\\label{eq:martingale_1}
	&\mcb N_t^{n,\beta}(f):=(\mcb M_t^{n,\beta}(f))^2 - \int_0^t n^\beta \mcb L_n (\langle \pi_{s}^{n,\beta}, f\rangle)^2 -2 \langle \pi_{s}^{n,\beta}, f\rangle\mcb L_n \langle \pi_{s}^{n,\beta}, f\rangle ds
\end{flalign}
are mean-zero martingales with respect to the natural filtration of the process $\{\eta_t\}_{t\geq 0}$. 
A simple computation shows that
\begin{equation}
\begin{split}
&\mcb L_n\eta(1)=(a_n+b_n+d_n)(\eta(2)-\eta(1))+c_n(\eta(n)-\eta(1))
\\& 
\mcb L_n\eta(2)=(a_n+b_n)(\eta(3)-\eta(2))+(c_n+d_n)(\eta(1)-\eta(2))\\& 
\mcb L_n\eta(n-1)=a_n(\eta(1)-\eta(n-1))++b_n(\eta(n)-\eta(n-1))+c_n(\eta(n-2)-\eta(n-1))\\&
\mcb L_n\eta(n)=b_n(\eta(1)-\eta(n))+c_n(\eta(n-1)-\eta(n))
\end{split}
\end{equation} and $\mcb L_n\eta(x)=(a_n+b_n)(\eta(x+1)-\eta(x))+c_n(\eta(x-1)-\eta(x))$ for all $x\in\{3,4,\cdots, n-2\}$.
After long, but elementary computations, we get,
\begin{equation}\label{martingale01}
\begin{split}
	n^{\beta} \mcb L_n \langle \pi_{s}^{n,\beta},f\rangle&=- (a_n+b_n)n^{\beta-2} \sum_{x \in \bb Z_n} \eta_{s n^{\beta}}(x) \nabla _n^- f(\tfrac{x}{n} )+ c_n n^{\beta-2} \sum_{x \in \bb Z_n} \eta_{s n^{\beta}}(x) \nabla_n^+ f(\tfrac{x}{n})\\
	&-a_n n^{\beta-2} \left( \eta_{s n^{\beta}}(1)-\eta_{s n^{\beta}}(n) \right) \nabla_n^- f\left(0 \right) + d_n  n^{\beta-2} \left( \eta_{s n^{\beta}}(1)-\eta_{s n^{\beta}}(2) \right) \nabla_n^+ f(\tfrac{1}{n} ).
\end{split}
\end{equation}
Above,  $\nabla_n^-$, $\nabla_n^+$  are  discrete operators defined on  $f: \bb T \to \bb R$ and  $x \in \bb Z_n$ by
$$
\nabla_n^-f(\tfrac{x}{n} )=n( f(\tfrac{x}{n}) -f(\tfrac{x-1}{n})) \quad \textrm{and} \quad \nabla_n^+f(\tfrac{x}{n})=\nabla_n^-f(\tfrac{x+1}{n}).$$

The quadratic variation of the martingale given in \eqref{eq:martingale} is equal to 
\begin{flalign}\label{qv01}
	\left\<  \mcb M^{n,\beta}(f) \right\>_t&= \int_0^t (a_n+b_n) n^{\beta-2} \left( \pi_{sn^{\beta}}^n(\nabla_n^{-} f) \right)^2 ds + \int_0^t c_n n^{\beta-2} \left( \pi_{sn^{\beta}}^n(\nabla_n^{+} f) \right)^2 ds\\
	&+\int_0^t 2 a_n n^{\beta-3} (\eta_{sn^{\beta}}(1)-\eta_{sn^{\beta}}(n)) \nabla_n^- f(0) \pi_{sn^{\beta}}^n(\nabla_n^{-} f)\, ds \nonumber\\
	& +\int_0^t a_n n^{\beta-4} (\eta_{sn^{\beta}}(1)-\eta_{sn^{\beta}}(n))^2 \left( \nabla_n^- f(0) \right)^2 ds \nonumber\\& + \int_0^t d_n n^{\beta-4} (\eta_{sn^{\beta}}(1)-\eta_{sn^{\beta}}(2))^2 ( \nabla_n^+ f(\tfrac{1}{n}))^2 ds. \nonumber
\end{flalign}
Now we prove the validity of (1). Recall  \eqref{psi} and \eqref{eq:norm-}. From Markov's inequality,
\begin{flalign}
 \bb P_{\mu_n} \Big(  \sup_{t \in [0,T]} \|\pi_{t}^{n,\beta}\|_{-m}>A \Big) \leq \frac{1}{A^2}    \sum_{k \in \bb Z} \gamma_k^{-m} \bb E_{\mu_n} \left[ \sup_{t \in [0,T]}\Big|\left\< \pi_{t}^{n,\beta}, \psi_k \right\>\Big|^2 \right]. \label{a0}
\end{flalign}
To estimate last identity we analyse (\ref{eq:martingale}) for each time scale. Note that  for $\beta=1$ and from a Taylor's expansion, we get
	\begin{flalign}\label{expre}
		\langle \pi_{t}^{n,1},\psi_k\rangle &=\mcb  M_t^{n,1}(\psi_k)+\langle \pi_0^{n,1},\psi_k\rangle   - (a_n+b_n-c_n)  \int_0^t  \langle \pi_{s }^{n,1},\nabla \psi_k\rangle ds +  \mcb O (n^{-1}),
	\end{flalign}
	while for $\beta=2$, we get
\begin{flalign}\label{expre2}
	\langle \pi_{t}^{n,2},\psi_k\rangle &=\mcb  M_t^{n,2}(\psi_k)+\langle \pi_0^{n,2},\psi_k\rangle   + c_n   \int_0^t \langle \pi_{s}^{n,2},\Delta_n \psi_k\rangle ds +\left(c_n- (a_n+b_n)\right)\int_0^t \langle \pi_{s}^{n,2},\nabla_n^- \psi_k\rangle ds\\
	& -a_n  \int_0^t ( \eta_{s n^{2}}(1)-\eta_{s n^{2}}(n) ) \nabla_n^- \psi_k(0) ds + d_n  \int_0^t ( \eta_{s n^{2}}(1)-\eta_{s n^{2}}(2)) \nabla_n^+ \psi_k(\tfrac{1}{n} ) ds. \nonumber
\end{flalign}
	Above  the discrete Laplacian operator is defined on $f: \bb T \to \bb R$ and  $x \in \bb Z_n$ by $$\Delta_nf(\tfrac{x}{n} )=n^2( f(\tfrac{x+1}{n}) +f(\tfrac{x-1}{n})-2f(\tfrac{x}{n})).$$
	By the convexity inequality $|\sum_{k=1}^\ell x_k|^2 \leq \ell \sum_{k=1}^\ell |x_k|^2$ for $\ell \in \bb N$, we  are left to estimate the contribution of each term in the previous identities \eqref{expre} and \eqref{expre2}. We note that the initial term can be estimated by the Cauchy-Schwarz inequality, as
	\begin{flalign}\label{a21}
			\bb E_{\mu_n}\Bigg[  |\langle \pi_0^{n,\beta},\psi_k\rangle |^2  \Bigg]\leq \frac{1}{n}\sum_{x \in \bb Z_n} |\psi_k(\tfrac x n)|^2.
	\end{flalign}
	Now we split the proof and we start with the hyperbolic case. 
	
From Doob's inequality, \eqref{qv01} for $\beta=1$ and noting that $\eta(x)\leq 1$ for all $x\in\mathbb Z_n$, we get
	\begin{flalign}
		\bb E_{\mu_n}&\left[ \sup_{t \in [0,T]} |\mcb M_t^{n,1}(\psi_k)|^2  \right] \leq 4   \bb E_{\mu_n}\Bigg[  |\mcb M_T^{n,1}(\psi_k)|^2  \Bigg] \lesssim n^{-1}. \label{a2}
	\end{flalign}
	Now, it remains to estimate the third term on the right hand-side of \eqref{expre}.
	By the definition of $\psi_k$ and using the Cauchy-Schwarz inequality twice, we get
	\begin{flalign}
		\bb E_{\mu_n}\Bigg[ \sup_{t \in [0,T]} \Big|\int_0^t  \langle \pi_{s }^{n,1},\nabla \psi_k \rangle ds\Big|^2  \Bigg] \leq T \, \bb E_{\mu_n}\Bigg[  \int_0^T\Big| \langle \pi_{s }^{n,1},\nabla \psi_k\rangle\Big|^2 ds  \Bigg] \leq  \frac{T^2}{n}\sum_{x \in \bb Z_n} |\nabla \psi_k(\tfrac xn)|^2. \label{a3}
	\end{flalign}
	From the previous results and since $\frac{1}{n}\sum_{x \in \bb Z_n} |f(\tfrac xn)|^2$ converges to $\left\< f,f \right\>$ for any $f \in C^\infty (\bb T)$, we have
\begin{flalign*}
\limsup_{n \to \infty} \bb P_{\mu_n} \left(  \sup_{t \in [0,T]} \|\pi_{t}^{n,1}\|_{-m}>A \right)
& \lesssim{ (T^2+1)}\,{A^2}    \sum_{k \in \bb Z} \gamma_k^{-m} \Big\{  \left\< \psi_k,\psi_k \right\>+ \left\< \nabla \psi_k, \nabla \psi_k \right\>\Big\}\\
&\lesssim{ (T^2+1) }\,{A^2}\Big(1+2 \sum_{k=1}^\infty \Big\{  1+4 \pi^2 k^2 \Big\}^{1-m}\Big).
\end{flalign*}
Therefore, (1) follows from the fact that the series above converges for $m>3/2$.\\

Now we focus on the case $\beta=2$. 
From Doob's inequality, \eqref{qv01} for $\beta=2$, noting that $\eta(x)\leq 1$ for all $x\in\mathbb Z_n$, and by  the Cauchy-Schwarz inequality, we get
\begin{flalign}
	\bb E_{\mu_n}\Bigg[ \sup_{t \in [0,T]} |\mcb M_t^{n,2}(\psi_k)|^2  \Bigg]& \leq 4   \bb E_{\mu_n}\Bigg[  |\mcb M_T^{n,2}(\psi_k)|^2  \Bigg] \nonumber \\& = 4   \int_0^T (a_n+b_n) \bb E_{\mu_n}\Bigg[ \langle \pi_{s}^{n,2},\nabla_n^{-} \psi_k\rangle^2\Bigg]  ds +4 \int_0^t c_n \bb E_{\mu_n}\Bigg[  \langle \pi_{s}^{n,2},\nabla_n^{+} \psi_k\rangle^2\Bigg] ds + \mcb O (n^{-1}) \nonumber \\
	&\lesssim\frac{ (a_n+b_n)T}{n} \sum_{x \in \bb Z_n}  |\nabla_n^{-} \psi_k (\tfrac xn)|^2 + \frac{ c_n T}{n} \sum_{x \in \bb Z_n}  |\nabla_n^{+} \psi_k (\tfrac xn)|^2. \label{b2}\end{flalign}
With the same arguments as used in  \eqref{a3} we see that 
\begin{flalign}
	\bb E_{\mu_n} \Bigg[ \sup_{t \in [0,T]}  \Big|\int_0^t \langle  \pi_{s}^{n,2},\Delta_n \psi_k\rangle ds \Big|^2   \Bigg] \leq   \frac{T^2}{n}\sum_{x \in \bb Z_n} |\Delta_n \psi_k(\tfrac xn)|^2,\label{b3}\\
	\bb E_{\mu_n} \Bigg[ \sup_{t \in [0,T]}  \Big|\int_0^t  \langle \pi_{s}^{n,2},\nabla_n^- \psi_k\rangle ds \Big|^2   \Bigg] \leq   \frac{T^2}{n}\sum_{x \in \bb Z_n} |\nabla_n^- \psi_k(\tfrac xn)|^2. \label{b4}
\end{flalign}
Finally, by the Cauchy-Schwarz inequality we have
\begin{flalign}
\bb E_{\mu_n} \Bigg[ \sup_{t \in [0,T]}  \Big|a_n  \int_0^t \left( \eta_{s n^{2}}(1)-\eta_{s n^{2}}(n) \right) \nabla_n^- \psi_k\left(0 \right) ds  \Big|^2   \Bigg] \leq |a_n|^2 T |\nabla_n^- \psi_k(0)|^2\label{b5},\\
\label{b6}
\bb E_{\mu_n} \Bigg[ \sup_{t \in [0,T]}  \Big|d_n  \int_0^t \left( \eta_{s n^{2}}(1)-\eta_{s n^{2}}(2) \right) \nabla_n^+ \psi_k\left(\tfrac{1}{n} \right) ds  \Big|^2   \Bigg] \leq |d_n|^2 T \left| \nabla_n^+ \psi_k(\tfrac{1}{n}) \right|^2.
\end{flalign}
By (\ref{a21}), (\ref{b2}), (\ref{b3}), (\ref{b4}), (\ref{b5}) and (\ref{b6}), we get
\begin{flalign*}
	\limsup_{n \to \infty}& \,\,\bb P_{\mu_n} \Bigg(  \sup_{t \in [0,T]} \|\pi_{tn^2}^n\|_{-m}>A \Bigg) \leq \frac{1}{A^2}    \sum_{k \in \bb Z} \gamma_k^{-m} \limsup_{n \to \infty} \bb E_{\mu_n} \Bigg[ \sup_{t \in [0,T]}|\pi_{tn^2}^n(\psi_k)|^2  \Bigg]\\
	&\lesssim \frac{1}{A^2}    \sum_{k \in \bb Z} \gamma_k^{-m} \Big\{  \left\< \psi_k,\psi_k \right\>+ \left\< \nabla \psi_k, \nabla \psi_k \right\>+ \left\< \Delta \psi_k, \Delta \psi_k \right\>\Big\} + \frac{(a^2+d^2)T}{A^2}    \sum_{k \in \bb Z} \gamma_k^{-m} |\nabla \psi_k (0)|^2\\
	&\lesssim\frac{1 }{A^2}+ \frac{1}{A^2}    \sum_{k=1}^\infty \Big\{  1+4 \pi^2 k^2+ (2\pi k)^4 \Big\} \left\{  1+ (2 \pi k)^2 \right\}^{-m} + \frac{ \pi^2 (a^2+d^2)T}{A^2}    \sum_{k =1}^\infty k^2\left\{  1+  \pi^2 k^2 \right\}^{-m}.
\end{flalign*}
Observe that the series above converge if $m>5/2$. This concludes the proof of (1).

Now we prove (2). Observe that, arguing similarly, one can show that if $m>3/2$ for $\beta=1$ or $m>5/2$ for $\beta=2$, then
\begin{flalign*}
	\lim_{N \to \infty} \limsup_{n \to\infty} \mathbb E_{\mu_n}  \Bigg[  \sup_{t \in [0,T]} \,\,\sum_{|k| \geq N} | \langle \pi_{t}^{n,\beta},\psi_k\rangle|^2 \gamma_k^{-m}   \Bigg]=0.
\end{flalign*}
Thus, (2) follows if for any integer $N$ and every $\varepsilon>0$,
$$
\lim_{\delta \to 0} \limsup_{n \to \infty} \bb P_{\mu_n}  \Bigg[   \sup_{\substack{|s-t|\leq \delta\\0\leq s, t \leq T}} \,\,\sum_{|k| \leq N} |\langle \pi_{t}^{n,\beta},\psi_k\rangle-\langle \pi_{s}^{n,\beta},\psi_k\rangle|^2 \gamma_k^{-m}\geq \varepsilon   \Bigg]=0.
$$
Observe also that to prove the above limit it is sufficient to show that
\begin{equation}\label{end_tight}
\lim_{\delta \to 0} \limsup_{n \to \infty} \bb P_{\mu_n}  \Bigg[   \sup_{\stackrel{|s-t|\leq \delta}{0\leq s, t \leq T}}| \langle \pi_{t}^{n,\beta},\psi_k\rangle-\langle \pi_{s}^{n,\beta},\psi_k\rangle|^2 \geq \varepsilon    \Bigg]=0
\end{equation}
for every $k \in \bb Z$. 	Now we split the proof and we start with the hyperbolic regime. 
Recall \eqref{qv01} for $\beta=1$. Then  the previous limit is a consequence of the next two results
 \begin{equation*}
 \begin{split}
		&(A)\quad \quad  \lim_{\delta \to 0} \limsup_{n \to \infty}  \bb P_{\mu_n} \Bigg[  \sup_{\substack{|s-t|\leq \delta\\0\leq s, t \leq T}} \left|\int_s^t \langle \pi_{r}^{n,1} ,\nabla \psi_k\rangle dr\right| > \varepsilon  \Bigg]=0.\\
&	(B)\quad \quad  \lim_{\delta \to 0} \limsup_{n \to \infty}  \bb P_{\mu_n} \Bigg[  \sup_{\substack{|s-t|\leq \delta\\0\leq s, t \leq T}} |\mcb M_t^{n,1}(\psi_k)-\mcb M_s^{n,1}(\psi_k)| > \varepsilon  \Bigg]=0.\end{split}
\end{equation*}
We start by proving (A). Note that from  Markov and Cauchy-Schwarz inequalities, and since $|t-s| \leq \delta$ and $s,t \leq T$, we get 
\begin{flalign*}
\bb P_{\mu_n} \Bigg[  \sup_{\substack{|s-t|\leq \delta\\0\leq s, t \leq T}} \Big|\int_s^t \langle \pi_{r}^{n,1} ,\nabla \psi_k\rangle dr\Bigg| > \varepsilon  \Big] 
 \leq  \frac{\delta T}{ \varepsilon^2} \bb E_{\mu_n} \Bigg[  \int_0^T |\langle\pi_{r}^{n,1} ,\nabla \psi_k\rangle|^2   dr \Bigg]\leq  \frac{\delta T^2}{ \varepsilon^2 n} \sum_{x \in \bb Z_n} |\nabla \psi_k (\tfrac xn)|^2.
\end{flalign*}

Now we prove (B). 
For $\beta>0$, define
\begin{flalign}\label{omega}
\omega_{\delta}(\mcb M^{n,\beta}(f))=\inf_{t_i} \max_{i \in [0,r)} \sup_{t_i \leq s < t <t_{i+1}} |\mcb M_t^{n,\beta}(f)-\mcb M_s^{n,\beta}(f)|,
\end{flalign}
where the first infimum is taken over all partitions of $[0,T]$ such that
$
0=t_0 < t_1 < \cdots <t_r=T$ and $
t_{i+1}-t_i>\delta$ for $0 \leq i <r.$
Since $$\sup_{t\in[0,T]} \Big|\mcb M_t^{n,\beta}(f)-\mcb  M_{t-}^{n,\beta}(f)\Big|=\sup_{t\in[0,T]} \Big|\langle \pi_{t}^{n,\beta},\psi_k\rangle-\langle \pi_{t^-}^{n,\beta},\psi_k\rangle\Big|\leq \frac{2}{n} \|\nabla \psi_k \|_{\infty}$$ and
$$
\sup_{\substack{|s-t|\leq \delta\\0\leq s, t \leq T}} \|\mcb M_t^{n,\beta}(\psi_k)-\mcb M_s^{n,\beta}(\psi_k)\|_{-m}  \leq 2 \omega_{\delta}(\mcb M^{n,\beta}(\psi_k)) + \sup_{t\in[0,T]} \Big|\mcb M_t^{n,\beta}(\psi_k)-\mcb M_{t-}^{n,1}(\psi_k)\Big|,
$$
the proof ends (for $\beta=1$) if we show that
$$
\lim_{\delta \to 0} \limsup_{n \to \infty} \bb P_{\mu_n} \Bigg[  \omega_{\delta}(\mcb M^{n,1}(\psi_k)) > \varepsilon \Bigg]=0
$$
for every $\varepsilon>0$. By \cite[Chapter 4, Proposition 1.6]{claudios}, it is enough to check that for every $\varepsilon>0$
$$
\lim_{\delta \to 0} \limsup_{n \to \infty} \sup_{\stackrel{\tau \in \mathcal T_T}{\theta \in [0,\delta]}} \bb P_{\mu_n} \Bigg[ \Big|\mcb M_{(\tau+\theta)\wedge T}^{n,1}(\psi_k)-\mcb M_{\tau}^{n,1}(\psi_k)\Big| >\varepsilon    \Bigg]=0,
$$
where $\mathcal T_T$ stands for all stopping times bounded from above by $T$. From Markov's inequality, since $\mcb M_t^{n,1}(\psi_k)$ is a martingale and since $\tau$ is a bounded stopping time,
\begin{flalign*}
\bb P_{\mu_n} \Bigg[ \Big|\mcb M_{\tau+\theta}^{n,1}(f)-\mcb M_{\tau}^{n,1}f)\Big| >\varepsilon    \Bigg] &\leq \frac{1}{\varepsilon^2} \bb E_{\mu_n} \Bigg[ \left\<\mcb M^{n,1}(f)\right\>_{\tau+\theta} -\left\<\mcb M^{n,1}(f)\right\>_{\tau}    \Bigg]\lesssim Cn^{-1}.
\end{flalign*}
And this ends the proof for the case $\beta=1$. For $\beta=2$, recall \eqref{qv01}.  
Then \eqref{end_tight} 
is a consequence of (A) (which also holds for $\beta=2$) and to prove (B), it is enough to note that 
	\begin{flalign*}
		\bb P_{\mu_n} &\Bigg[ \Big|\mcb M_{\tau+\theta}^{n,2}(f)-\mcb M_{\tau}^{n,2}(f)\Big| >\varepsilon    \Bigg] \\&\leq \frac{1}{\varepsilon^2} \bb E_{\mu_n}\Bigg[ \int_\tau^{\tau+\theta} \!\!\!\!(a_n+b_n)  (\langle \pi_{s}^{n,2},\nabla_n^{-} \psi_k \rangle )^2 ds + \int_\tau^{\tau+\theta} c_n  (\langle \pi_{s}^{n,2},\nabla_n^{+} f\rangle)^2 ds  \Bigg] + \mcb O (n^{-1})\\
		&\leq \frac{\delta}{\varepsilon^2}\Big[(a_n+b_n)\|\nabla_n^- f\|_{\infty}^2 + c_n\|\nabla_n^+ f\|_{\infty}^2 \Big] + \mcb O (n^{-1})
	\end{flalign*}
	and last display vanishes by taking $n\to+\infty$ and then $\delta\to 0. 
$
This ends the proof of Lemma \ref{criterion} for $\beta\in\{1,2\}$.

From last results, we  know that $\{\bb Q_{n,\beta}\}_{ n \in \bb N }$ is tight in the weak topology of $\mcb D([0,T],\mcb H_{-m})$ for any $m>3/2$ when $\beta=1$ and $m>5/2$ for $\beta=2$.  Now we characterize the limit point.

\subsection{Characterization of limit points}

Now we complete the proof of Theorems~\ref{hydrodynamic limit 1} and \ref{hydrodynamic limit 3}.
From the computations of last subsection, we know that there exists a limiting point  $\mathbb Q_{\beta}$ of  the sequence $\{\bb Q_{n,\beta}\}_{ n \in \bb N }$. Since in both cases of $\beta$ we have $\eta(x)\leq 1$, then  $\mathbb Q_{\beta}$  is supported on trajectories of measures that are absolutely continuous with respect to the Lebesgue measure: $$\bb Q_\beta(\pi_\cdot: \forall t\in[0,T], \,\pi_t(du)=\varrho(t,u)du)=1.$$
Now we characterize the density. 
We start with the hyperbolic case and we present an heuristic argument, from which one easily deduces the rigorous one. 
Recall \eqref{martingale01} with $\beta=1$. 
From the previous section we know that the martingale \eqref{eq:martingale} vanishes in $\mathbb L^2(\mathbb P_{\mu_n})$, as $n\to+\infty$.   Since the limit $\pi^1_t(du)$ satisfies $\pi^1_t(du)=\varrho^\kappa(t,u)du$ and $\mu_n $ is associated to $\varrho_0$,  we get
$$
	0=\langle \varrho^\kappa_t,f\rangle -\langle \varrho_0,f\rangle + \int_0^t  (a+b-c)\langle \varrho_s^\kappa,\nabla  f\rangle ds
$$
which matches with the notion of weak solution given in  Definition \ref{def_transport} with $\mathfrak g=\varrho_0$ and  $\kappa=a+b-c$. 
Now we go to the diffusive case. Recall \eqref{martingale01} with $\beta=2$. By reorganizing terms and using the fact that $a_n+b_n-c_n=\tfrac{\gamma}{n}$, we can rewrite
\begin{flalign}\label{use_pat}
	\mcb M_t^{n,2}(f)&=\langle \pi_{t}^{n,2},f \rangle-\langle \pi_{0}^{n,2}, f\rangle- c_n   \int_0^t \langle \pi_{s}^{n,2},\Delta_n f\rangle ds +\gamma \int_0^t \langle\pi_{s}^{n,2},\nabla_n^- f\rangle ds\\
	& +a_n  \int_0^t ( \eta_{s n^{2}}(1)-\eta_{s n^{2}}(n) ) \nabla_n^- f(0) ds - d_n  \int_0^t ( \eta_{s n^{2}}(1)-\eta_{s n^{2}}(2) ) \nabla_n^+ f(\tfrac{1}{n} ) ds. \nonumber
\end{flalign}
 In this case, the argument is similar to one used in  $\beta=1$, but now the martingale  $\mcb M_t^{n,2}(f)$ does not vanish, and showing that the terms in the last line of last display vanish is more complicated since we need to use the dynamics of the process. To overcome this problem we use the two lemmas below, which allows to conclude that in the limit, as $n\to+\infty$, we obtain
\begin{equation*}
		\mcb M^2_t(f) := \langle \pi^2_t,f\rangle-\langle\pi^2_0, f\rangle -c\int_0^t \langle \pi^2_s,\Delta f\rangle ds +\gamma\int_0^t \langle \pi^2_s,\nabla f\rangle ds.
		\end{equation*}
The next lemma is necessary to control the boundary terms in \eqref{use_pat}. 
\begin{lemma}\label{replacement1}
	Let $r\in\{2,n\}$ and assume that if $r=2$ then $d_n>0$ or if $r=n$ then $b_n>0$.  For any $t >0$ and $\delta\in(0,1)$
	\begin{flalign*}
		\bb E_{\mu_n}\Bigg[ \Big| \int_0^t (\eta_{sn^2}(1)-\eta_{sn^2}(r)) ds \Big| \Bigg] \lesssim t\max\Big\{\frac{1}{n^{\delta}},\frac{1} {\mathfrak C_r n^{1-\delta}}\Big\},
	\end{flalign*}
	where $~\mathfrak C_r=d_n\textbf{1}_{r=2}+b_n \textbf{1}_{r=n}$.
\end{lemma}
The next lemma is necessary to control the martingale appearing in \eqref{use_pat}, which, contrarily to the case $\beta=1$ does not vanish.
\begin{lemma}\label{martingale limit}
 The sequence of martingales $\{ \mcb M_t^{n,2}(f); t \geq 0 \}_{n \in \bb N}$ introduced in \eqref{eq:martingale} with $\beta=2$, converges, in distribution as $n\to+\infty$, to a mean-zero martingale $\mcb M^2_t(f)$ with quadratic variation given by
	$$
	\left\< \mcb M^2(f) \right\>_t:=2c\int_0^t \langle \pi^2_s,\nabla f\rangle ^2 ds.
	$$
\end{lemma}
\begin{proof}[Proof of Lemma~\ref{replacement1}]
First observe that $\mu_n$ is a general measure  on $\Omega_n$ and by the entropy inequality we can pay a price to replace it by the stationary measure, the Bernoulli product measure $\nu_\varrho$. A simple computation shows that  there exists a constant $C=C(\varrho)$ such that 
\begin{flalign}\label{remark entropy}
	H(\mu|\nu_\varrho)= \sum_{\eta \in \Omega_n} \mu (\eta) \log{\Big( \frac{\mu(\eta)}{\nu_\varrho(\eta)} \Big)}\leq C\,n.
\end{flalign}
This together with  Jensen's inequality, gives for any $A>0$ that
	\begin{flalign}\label{d2}
		\bb E_{\mu_n}\Bigg[ \Big| \int_0^t (\eta_{sn^2}(1)-\eta_{sn^2}(r)) ds \Big| \Bigg] \leq \frac CA + \frac{1}{A\,n} \log \bb E_{\nu_\varrho}\Bigg[ \exp \Big\{ A\,n\,\Big| \int_0^t (\eta_{sn^2}(1)-\eta_{sn^2}(r)) ds \Big|  \Big\} \Bigg].
	\end{flalign}
Since $e^{|x|}\leq e^{x}+e^{-x}$ and
	$$
	\limsup_{n \to \infty}  \log (B_n + C_n)=\max \left\{  \limsup_{n \to \infty}  \log (B_n)\,,\,\limsup_{n \to \infty} \log (C_n) \right\},
	$$
	we can remove the absolute value on the right-hand side of (\ref{d2}). From  Feynman-Kac's formula,  the  rightmost term  of \eqref{d2} is bounded from above by
	\begin{flalign}\label{d1}
 \int_0^t \, \sup_{f}\,\left\{ \int  (\eta(1)- \eta(r))\,f(\eta)  \,\nu_\varrho\,(d\eta) - \frac{n}{A}\langle \mcb L_n\sqrt f,\sqrt f\rangle_{\nu_\varrho} \right\} ds,
	\end{flalign}
	where the supremum runs over all densities $f$ with respect to $\nu_\varrho$. Now we introduce 
	the quadratic operator  given on  $g:\Omega_n\to\bb R$ by 
\begin{flalign}\label{dirichlet form}
	\mf D_n\,(g;\nu_\varrho) &= \frac{a_n}{2} \int \left( \sqrt{ g(\eta^{1 \rightarrow n-1})}-\sqrt{  g(\eta)}\right)^2 d\nu_\varrho +  \frac{b_n}{2} \int \left(\sqrt { g(\eta^{1 \rightarrow n})}-\sqrt{ g(\eta)}\right)^2 d\nu_\varrho\\& +  \frac{c_n}{2} \int \left(\sqrt{ g(\eta^{n \rightarrow 1})}-\sqrt{ g(\eta)}\right)^2 d\nu_\varrho+ \frac{d_n}{2}\int \left(\sqrt{ g(\eta^{1 \leftrightarrow 2})}-\sqrt{ g(\eta)}\right)^2 d\nu_\varrho\nonumber.
\end{flalign}
A simple computation shows that 
	the  Dirichlet form associated with the generator $ \mcb  L_n$, i.e.~ $\langle \mcb L_n g,g\rangle_{\nu_\varrho}$ is given by 	$\mf D_n\,(g;\nu_\varrho)$. 
Now, we write the leftmost term in the supremum in \eqref{d1} as one-half plus one-half of itself, and in one of the parcels we make  the change of variables $\eta \mapsto \eta^{1 \rightarrow r}$ for $r=n$ or  $\eta \mapsto \eta^{1 \leftrightarrow r}$ for $r=2$, and note that the probability measure $\nu_\varrho$ is invariant for both transformations. Then, from the elementary identity $(x-y)=(\sqrt{x}-\sqrt{y})(\sqrt{x}+\sqrt{y})$ and  Young's  inequality, we can bound that term from above by
	\begin{flalign*}
		\frac{B}{2}\,\int  \eta(1)\,\left( \sqrt{f(\eta^{1 \rightarrow n})}+\sqrt{f(\eta)}   \right)^2  d\nu_\varrho+\frac{1}{2B}\,\int  \left( \sqrt{f(\eta^{1 \rightarrow n})}-\sqrt{f(\eta)}  \right)^2  d\nu_\varrho,
	\end{flalign*}
	for $r=n$ and a similar expression for $r=2$. Above $B>0$. 
Recalling (\ref{dirichlet form}), using the fact that $f$ is a density with respect to $\nu_\varrho$ and choosing $	B=A/(b_n n)$, last display is bounded from above by 
$$
		\frac{2A}{b_n n} + \frac{n}{A} \, \mf D_n \,(\sqrt{f};\nu_\varrho).
$$
Taking $A=n^\delta$ the proof ends. For $r=2$, we repeat the same strategy as above and we take  instead $B=A/(d_n n)$.
\end{proof}

In order to prove Lemma~\ref{martingale limit}, we invoke \cite[Theorem VIII, 3.11]{jacod}, which we state in our setting as:

\begin{theorem}\label{jacodtheo}
	Let $\{M_t^n \,; \,t \in [0,T]\}_{ n \in \bb N }$ be a sequence of martingales in $\mcb D ([0,T],\bb R)$ with quadratic variation $\left\< M^n\right\>_t$. Assume that
	\begin{enumerate}
		\item For any $n \in \bb N$ the quadratic variation process $\left\< M^n\right\>_t$ has continuous trajectories almost surely;\vspace {0,2cm}
		\item $\lim_{n \to \infty} \bb E \Bigg[ \sup_{t \in [0,T]} | M_t^n - M_{t^-}^n | \Bigg]=0$;\vspace {0,2cm}
		\item $\left\< M^n\right\>_t$ converges to $\left\< M\right\>_t $, in distribution, as $n \to \infty$.\vspace {0,2cm}
	\end{enumerate}
	Then,  $\{M_t^n \,; \,t \in [0,T]\}_{ n \in \bb N }$ converges, in distribution, to  a mean zero martingale $\{M_t \}_{t \in [0,T] }$ with continuous trajectories and with quadratic variation  $\left\< M\right\>_t$.
\end{theorem}

\begin{proof}[Proof of Lemma~\ref{martingale limit}]	
We apply Theorem~\ref{jacodtheo} to the sequence $\{\mcb M_t^{n,2}(f);\, t \geq 0\}_{ n \in \bb N}$ of martingales given in (\ref{eq:martingale}) with $\beta=2$.  
From \eqref{eq:martingale_1} the condition (1) is trivially satisfied.   In order to verify (2), observe that at each time there is at most one type of change, either $\eta^{1 \rightarrow n-1}$, or $\eta^{1 \rightarrow n}$ or $\eta^{n \rightarrow 1}$ or $\eta^{1 \leftrightarrow 2}$.  Let us suppose that at time $t
$ it happens $\eta^{1 \rightarrow n}$, the others can be treated analogously. Note that in this case, we have $\eta_t(x-1)=\eta_{t^-}(x)$. Therefore, for  any function $f \in C^\infty(\bb T)$ we have that
	\begin{flalign*}
		\Big|\mcb M_t^{n,2}(f)-\mcb M_{t^-}^{n,2}(f)\Big| &= \Big| \pi_{t}^{n,2} (f) -  \pi_{t^-}^{n,2} (f)\Big|=\Big| \frac{1}{n^2}\sum_{x\in\mathbb Z_n}\nabla_n^+f(\tfrac xn)\eta_{tn^2}(x)\Big|\leq \frac{\| \nabla f \|_{\infty}}{n}, 
	\end{flalign*}
	which trivially implies (2).
Finally,  note that for $\beta=2$, the quadratic variation in \eqref{eq:martingale_1} is given by
\begin{flalign}\label{qv02}
	\left\<  \mcb M^{n,2}(f) \right\>_t&= \int_0^t (a_n+b_n)  \langle \pi_{s}^{n,2},\nabla_n^{-} f\rangle ^2 ds + \int_0^t c_n  \langle  \pi_{s}^{n,2},\nabla_n^{+} f\rangle ^2 ds\\
	&+\int_0^t 2 a_n n^{-1} (\eta_{sn^{2}}(1)-\eta_{sn^{2}}(n)) \nabla_n^- f(0) \langle \pi_{s}^{n,2},\nabla_n^{-} f\rangle ds \nonumber\\
	& +\int_0^t a_n n^{-2} (\eta_{sn^{2}}(1)-\eta_{sn^{2}}(n))^2 \left( \nabla_n^- f(0) \right)^2 ds \nonumber\\& + \int_0^t d_n n^{-2} (\eta_{sn^{2}}(1)-\eta_{sn^{2}}(2))^2 \left( \nabla_n^+ f\left(\tfrac{1}{n}\right) \right)^2 ds. \nonumber
\end{flalign}
From last identity, condition  (3) follows from the fact that $\{\pi_t^2\}_{t\geq 0}$ is a limit point, in distribution, of $ \{\pi_{t}^{n,2}\}_{t\geq 0}$. 
\end{proof}

\section{Equilibrium fluctuations }\label{s5}

In this section we prove Theorems~\ref{clt1} and \ref{clt 3}. 

\subsection{Tightness}\label{sec:tigh_flu}
We start by showing tightness of $\{\widetilde {\bb Q}_{n,\beta}\}_{n\in\bb N} $. As above we use Lemma \ref{criterion} and we start by showing (1). Note that  by Dynkin's formula, there exist two mean-zero martingales $\{\widetilde{\mcb M}_t^{n,\beta}(f)\}_{t\geq 0}$ and $\{\widetilde{\mcb N}_t^{n,\beta}(f)\}_{t\geq 0}$ as given in \eqref{eq:martingale} and \eqref{eq:martingale_1}, respectively. We start the proof of (1) in Lemma \ref{criterion} for  the hyperbolic case, i.e.~  $\beta=1$. Then, for $f \in C^{\infty}( \bb T)$ we obtain
\begin{flalign}\label{martingale0f}
\widetilde{\mcb M}_t^{n,1}(f)&=\mcb Y_{t}^{n,1}(f)-\mcb Y_{0}^{n,1}(f) 
	+ (a_n+b_n)  \int_0^t \mcb Y_{s}^{n,1}(\nabla_n^- f) ds -c_n  \int_0^t \mcb Y_{s}^{n,1}(\nabla_n^+ f) ds\\ & +a_n n^{-1/2} \int_0^t \left( \eta_{s n}(1)-\eta_{s n}(n) \right) \nabla_n^- f\left(0 \right) ds - d_n n^{-1/2} \int_0^t \left( \eta_{s n}(1)-\eta_{s n}(2) \right) \nabla_n^+ f\left(\tfrac{1}{n} \right)ds,\nonumber
\end{flalign}
and
$\widetilde{\mcb N}_t^{n,1}(f):=\left(\widetilde{\mcb M}_t^{n,1}(f)\right)^2 - \left\<  \widetilde{\mcb M}^{n,1}(f)\right\>_t,
$
where
\begin{flalign}\label{qv20}
\left\<  \widetilde{\mcb M}^{n,1}(f)\right\>_t&=\int_0^t (a_n+b_n) n^{-1} \left( \mcb Y_{s}^{n,1}(\nabla_n^{-} f) \right)^2 ds + \int_0^t c_n n^{-1} \left( \mcb Y_{s}^{n,1}(\nabla_n^{+} f) \right)^2 ds \\
	&+ \int_0^t 2 a_n n^{-2} (\eta_{sn}(1)-\eta_{sn}(n)) \nabla_n^- f(0) \mcb Y_{s}^{n,1}(\nabla_n^{-} f) ds \nonumber\\
	& + \int_0^t a_n n^{-2} (\eta_{sn}(1)-\eta_{sn}(n))^2 \left( \nabla_n^- f(0) \right)^2 ds \nonumber\\& +  \int_0^t d_n n^{-2} (\eta_{sn}(1)-\eta_{sn}(2))^2 \left( \nabla_n^+ f\left(\tfrac{1}{n}\right) \right)^2 ds. \nonumber
\end{flalign}
Recall  $\psi_k$ defined in (\ref{psi}). To prove  (1) we note that by Doob's inequality,  last identity and the fact that $\nu_\varrho$ is invariant, we get 
\begin{flalign*}
	\bb E_{\nu_\varrho}&\left[ \sup_{t \in [0,T]} |\widetilde{\mcb M}_t^{n,1}(\psi_k)|^2  \right] \leq 4   \bb E_{\nu_\varrho}\Bigg[  |\widetilde{\mcb M}_T^{n,1}(\psi_k)|^2  \Bigg] \nonumber  \\& = 4 \bb E_{\nu_\varrho}\left[  \int_0^T (a_n+b_n) n^{-1} \left( \mcb Y_{s}^{n,1}(\nabla_n^{-} \psi_k) \right)^2 ds + \int_0^t c_n n^{-1} \left(\mcb Y_{s}^{n,1}(\nabla_n^{+} \psi_k) \right)^2 ds \right] +\frac{(a_n+d_n)k^2 T}{n^2}=\mcb O(n^{-1}). 
\end{flalign*}
A simple computation shows that 
\begin{flalign*}
\bb E_{\nu_\varrho}\Bigg[  |\mcb Y_{0}^{n,1}(\psi_k) |^2  \Bigg]= \frac{\varrho(1-\varrho)}{n}\sum_{x \in \bb Z_n} |\psi_k(\tfrac xn)|^2.
\end{flalign*}
Finally, 
\begin{flalign*}
	\bb E_{\nu_\varrho}\Bigg[ \sup_{t \in [0,T]} \Big|\int_0^t  \mcb Y_{s}^{n,1}(\nabla \psi_k) ds\Big|^2  \Bigg] \leq T \, \bb E_{\nu_\varrho}\Bigg[  \int_0^T\Big| \mcb Y_{s}^{n,1}(\nabla \psi_k)\Big|^2 ds  \Bigg] \leq  \frac{\varrho(1-\varrho)T^2}{n}\sum_{x \in \bb Z_n} |\nabla \psi_k(\tfrac xn )|^2. 
\end{flalign*}
From the previous results we can easily show that 
\begin{flalign*}
	\limsup_{n \to \infty} \bb P_{\nu_\varrho} \left(  \sup_{t \in [0,T]} \|\mcb Y_t^{n,1}\|_{-m}>A \right)& \leq \frac{T^2+1  }{A^2}\Big(1+ 2    \sum_{k=1}^\infty \Big\{  1+4 \pi^2 k^2 \Big\}^{1-m}\Big).
\end{flalign*}
Note that the series above converges for $m>3/2$.\\

Now we prove (1) for $\beta=2$. Note that in this case
\begin{flalign}\label{martingalef}
 \widetilde{\mcb M}_t^{n,2}(f)&=\mcb Y_{t}^{n,2}(f)-\mcb Y_0^n(f) - c_n   \int_0^t \mcb Y_0^{n,2}(\Delta_n f) ds -\left(c_n- (a_n+b_n)\right)n \int_0^t \mcb Y_{s}^{n,2}(\nabla_n^- f) ds\\
	& +a_n \sqrt{n} \int_0^t \left( \eta_{s n^{2}}(1)-\eta_{s n^{2}}(n) \right) \nabla_n^- f\left(0 \right) ds - d_n \sqrt{n} \int_0^t \left( \eta_{s n^{2}}(1)-\eta_{s n^{2}}(2) \right) \nabla_n^+ f\left(\tfrac{1}{n} \right)ds, \nonumber
\end{flalign}
and the quadratic variation equals to 
\begin{flalign}\label{qv2_flu}
\left\<   \widetilde{\mcb M}^{n,2}(f) \right\>_t&
=\int_0^t (a_n+b_n)  \left( \mcb Y_{s}^{n,2}(\nabla_n^{-} f) \right)^2 ds + \int_0^t c_n  \left( \mcb Y_{s}^{n,2}(\nabla_n^{+} f) \right)^2 ds \\
&+ \int_0^t 2 a_n n^{-1} (\eta_{sn^{2}}(1)-\eta_{sn^{2}}(n)) \nabla_n^- f(0) \mcb Y_{s}^{n,2}(\nabla_n^{-} f) ds \nonumber\\
& + \int_0^t a_n n^{-1} (\eta_{sn^{2}}(1)-\eta_{sn^{2}}(n))^2 \left( \nabla_n^- f(0) \right)^2 ds \nonumber\\& +  \int_0^t d_n n^{-1} (\eta_{sn^{2}}(1)-\eta_{sn^{2}}(2))^2 \left( \nabla_n^+ f\left(\tfrac{1}{n}\right) \right)^2 ds. \nonumber
\end{flalign}

From Doob's inequality, (\ref{qv2_flu}) and the fact that $\nu_\varrho$ is invariant, we get  that 
\begin{flalign}
	\bb E_{\nu_\varrho}&\Bigg[ \sup_{t \in [0,T]} \Big|\widetilde{\mcb M}_t^{n,2}(\psi_k)\Big|^2  \Bigg] \lesssim \frac{ (a_n+b_n)T}{n} \sum_{x \in \bb Z_n}  |\nabla_n^{-} \psi_k (\tfrac xn)|^2 + \frac{c_n T}{n} \sum_{x \in \bb Z_n}  |\nabla_n^{+} \psi_k (\tfrac xn)|^2 +  \frac{ (a_n+d_n) T k^2} {n}. \label{d200}
\end{flalign}
The initial condition can be estimated as in the previous case. Moreover, as above we have 
\begin{flalign}\mathbb E_{\nu_\varrho}\Bigg[ \sup_{t \in [0,T]} \Big|\int_0^t  \mcb Y_{s}^{n,2}(\nabla_n^- \psi_k) ds\Big|^2  \Bigg] \leq  \frac{\varrho(1-\varrho)T^2}{n}\sum_{x \in \bb Z_n} |\nabla_n^- \psi_k(\tfrac xn)|^2.\label{d4} \end{flalign}
 The contribution of the boundary terms in  \eqref{martingalef} is estimated in the next lemma whose proof we present at the end of the proof of (1).

\begin{lemma}\label{replacement21}
	Let $r\in\{2,n\}$ and assume that if $r=2$ then $d_n>0$ or if $r=n$ then $b_n>0$.  For any $t \in [0,T]$ we have
	\begin{flalign*}
		\bb E_{\nu_\varrho} \Bigg[ \sup_{t \in [0,T]}\Big|   \sqrt{n} \int_0^{t} \left( \eta_{s n^{2}}(1)-\eta_{s n^{2}}(r) \right)   ds \Big|^2 \Bigg]\lesssim T\frac{1} { \mathfrak C_r n},		
	\end{flalign*}
	where we recall that $~\mathfrak C_r=d_n\textbf{1}_{r=2}+b_n \textbf{1}_{r=n}$.	
\end{lemma}
From the previous arguments, we conclude that 
\begin{flalign*}
	\limsup_{n \to \infty}\, \bb P_{\nu_\varrho} \Big(  \sup_{t \in [0,T]} \| \mcb Y_{t}^{n,2}\|_{-m}>A \Big) \leq \frac{C }{A^2}\Big( 1+2    \sum_{k=1}^\infty \Big\{  1+4\pi^2 k^2+ (2\pi k)^4 \Big\} \Big\{  1+ 4 \pi^2 k^2 \Big\}^{-m}\Big).
\end{flalign*}
The series above converges if $m>5/2$. This ends the proof of (1) in both cases $\beta\in\{1,2\}$. Now we prove Lemma \ref{replacement21}.

\begin{proof}[Proof of Lemma~\ref{replacement21}]
By Kipnis-Varadhan inequality (see~\cite[Lemma 4.3]{chang}), the expectation in the statement of the lemma 
is bounded from above by
\begin{flalign*}
T \sup_{f \in L^2(\nu_\varrho)} \left\{  \sqrt{n} \int (\eta(1)-\eta(r) )f(\eta) d\nu_\varrho- n^2 \mathfrak D_n (f;\nu_\varrho)  \right\}.
\end{flalign*} Above we used the operator $\mathfrak D_n $ introduced in
\eqref{dirichlet form}.
As in the proof of Lemma \ref{replacement1}, we can rewrite, for $r=n$,  the previous expression as
\begin{flalign*}
	T \sup_{f \in L^2(\nu_\varrho)} \left\{  \sqrt{n}\int \eta(1)\, \left(f(\eta) - f(\eta^{1 \rightarrow n})\right) \,  d\nu_\varrho - n^2 \mathfrak D_n (f;\nu_\varrho)  \right\}.
\end{flalign*}
For $r=2$ we get exactly the previous  expression but with the transformation $\eta \mapsto \eta^{1 \leftrightarrow 2}$. From Young's inequality, the expression inside the supremum above is bounded by
\begin{flalign*}
 \frac{A\sqrt{n}}{2} &\int  \left(f(\eta) - f(\eta^{1 \rightarrow n})\right)^2 \, d \nu_\varrho+ \frac{\sqrt{n}\varrho}{2A}  \,\,- n^2 \mathfrak D_n (f;\nu_\varrho),
\end{flalign*}
and for the choice $A= b_n n^{3/2}$ the proof ends.  For $r=2$ we make the choice $A=d_n n^{3/2}$.
\end{proof}

Now we prove (2) of Lemma \ref{criterion}. In both cases, it is enough to show that 
\begin{equation}
\lim_{\delta \to 0} \limsup_{n \to \infty} \bb P_{\nu_\varrho} \Bigg[   \sup_{\substack{|s-t|\leq \delta\\0\leq s, t \leq T}}\Big| \mcb Y_{t}^{n,\beta}(\psi_k)-\mcb Y_{t}^{n,\beta}(\psi_k)\Big|^2 \geq \varepsilon   \Bigg]=0
\end{equation}
for every $k \in \bb Z$. 
Now we start with the case $\beta=1$. Last limit is  a consequence of the next two results.
	For any function $f \in C^\infty(\bb T)$ and any $\varepsilon>0$ we have
	\begin{flalign*}
\bb P_{\nu_\varrho} \Bigg[  \sup_{\substack{|s-t|\leq \delta\\0\leq s, t \leq T}} \Bigg|\int_s^t \mcb Y_s^{n,1} (\nabla f) dr\Big| > \varepsilon \Bigg]  \leq  \frac{\delta}{ \varepsilon^2} \bb E_{\nu_\varrho} \Bigg[  \int_0^T \Big|\mcb Y_s^{n,1}  (\nabla f)\Bigg|^2   dr \Bigg]= \frac{\varrho(1-\varrho)\delta T}{ \varepsilon^2 n} \sum_{x \in \bb Z_n} |\nabla f (\tfrac xn)|^2.
	\end{flalign*}
Moreover, for any function $f \in C^\infty(\bb T)$ and any $\varepsilon>0$ we have
	$$
	\lim_{\delta \to 0} \limsup_{n \to \infty}  \bb P_{\nu_\varrho} \Bigg[  \sup_{\substack{|s-t|\leq \delta\\0\leq s, t \leq T}} \Big|\widetilde{\mcb M}_t^{n,1}(f)-\widetilde{\mcb M}_s^{n,1}(f)\Big| > \varepsilon  \Bigg]=0.
	$$
Last result is a consequence of the fact that
	\begin{flalign*}
		\bb P_{\nu_\varrho} \Bigg[ \Big|\widetilde{\mcb M}_{\tau+\theta}^{n,1}(f)-\widetilde{\mcb M}_{\tau}^{n,1}(f)\Big| >\varepsilon    \Bigg] &\leq \mcb O(n^{-1}).
	\end{flalign*}

This proves (2) for the case $\beta=1$. Now we analyse (2) for $\beta=2$, which is a consequence of the next results.
For any function $f \in C^\infty(\bb T)$ and any $\varepsilon>0$ we have
	$$
	\lim_{\delta \to 0} \limsup_{n \to \infty}  \bb P_{\nu_\varrho} \Bigg[  \sup_{\substack{|s-t|\leq \delta\\0\leq s, t \leq T}} \Bigg|\int_s^t \mcb Y_{t}^{n,2} (\nabla_n^- f) dr\Bigg| > \varepsilon  \Bigg]=0,
	$$
	And the same limit above holds for $\Delta_n f$ instead of $\nabla_n^- f$.
	From a similar argument to the one used to prove Lemma \ref{replacement21}
	we get for $r\in\{2,n\}$ that
	$$
	\lim_{\delta \to 0} \limsup_{n \to \infty}  \bb P_{\nu_\varrho} \Bigg[  \sup_{\substack{|s-t|\leq \delta\\0\leq s, t \leq T}} \Big|   \sqrt{n} \int_s^t \left( \eta_{r n^{2}}(1)-\eta_{r n^{2}}(r) \right) dr   \Big| > \varepsilon  \Bigg]=0.
	$$
Finally, as above it remains to check that for every $\varepsilon>0$
	$$
	\lim_{\delta \to 0} \limsup_{n \to \infty} \sup_{\substack{\tau \in \mathcal T_T\\\theta \in [0,\delta]}} \bb P_{\nu_\varrho} \Bigg[\Big |\widetilde{\mcb M}_{(\tau+\theta)\wedge T}^{n,2}(f)-\widetilde{\mcb M}_{\tau}^{n,2}(f)\Big| >\varepsilon    \Bigg]=0, 
	$$
	 and as in the hyperbolic case it is easy to check that 
	\begin{flalign*}
		\bb P_{\nu_\varrho} &\Bigg[\Big |\widetilde{\mcb M}_{(\tau+\theta)\wedge T}^{n,2}(f)-\widetilde{\mcb M}_{\tau}^{n,2}(f)\Big| >\varepsilon \Bigg] \lesssim \frac{\delta}{\varepsilon^2}\varrho(1-\varrho)\Bigg[(a_n+b_n)\|\nabla_n^- f\|_{\infty}^2 + c_n\|\nabla_n^+ f\|_{\infty}^2 \Bigg] +   \frac{(a_n+d_n)k^2 T}{n}.
	\end{flalign*}
	And this ends the proof of tightness.

\subsection{Characterization of the limit points}

In order to conclude the proofs of Theorems~\ref{clt1} and \ref{clt 3}, we  need to show that the limit points are solutions of the martingale problem given in the statements of the results. The hyperbolic case is trivial and we leave the details to the reader. We just note that the limit field satisfies
$$
	0=\mcb Y^1_t(f)-\mcb Y^1_0(f) + \int_0^t  (a+b-c)\mcb Y^1_s(\nabla  f) ds
$$
for any $f\in C^1(\bb T)$
and from this identity it follows that the covariance is given by \eqref{cov_field}.

 In the diffusive case, the result follows as a direct consequence of the expressions given in (\ref{martingalef}) and (\ref{qv2_flu}), Lemmas   \ref{replacement21} and \ref{martingale limit}.

\section{Proof of Theorem \ref{theo:uni_mp}}\label{sec5}

In order to prove the theorem, we   show the uniqueness of strong solutions of
\eqref{arica} and then, we prove that every solution of the martingale problem is a strong solution. Since all the strong solutions have the same, we are done. 

 We define a strong solution to \eqref{arica} in the following way. 
We say that a process $\{X_t\}_{t\geq 0}$ in $ \mcb C([0,T], \mcb H_{-m})$ is a strong solution of \eqref{arica} if for all $f\in\mcb H_{m}$ it holds
\begin{equation}
\label{iquique}
d\langle X_t,f\rangle =\langle X_t, (\nu\Delta-\mcb v\nabla )f\rangle dt-\sigma \langle X_t,\nabla f\rangle d \mcb B_t.
\end{equation}
In order to prove uniqueness of strong solutions, first, recall that the functions $g_k(u)=e^{ 2\pi i k u }$, $u \in \bb T,k \in \bb Z$, form an orthonormal basis of $L^2(\bb T)$. For each $k \in \bb Z$, let $\widehat f(k)=\left\< f,g_k \right\>$ and let us rewrite $X_t$ as a linear combination of the aforementioned eigenfunctions:
\begin{flalign*}
X_t&= \sum_{k \in \bb Z}\widehat X_t(k)\,g_k.
\end{flalign*}
Observe that for each $k$ the process $\widehat X_t(k)$ is solution of
$$
d \widehat X_t(k)=-\lambda_k \widehat X_t(k) dt  +\theta i \widehat X_t(k) d {\mcb  B}_t,
$$
where $\lambda_k=\nu(2 \pi k)^2-2 \pi \mcb v k i$ and $\theta_k=2\pi  k\sigma $. As a consequence, it is enough to show uniqueness of the strong solution of the previous equation.  Let  $Y_t(k):=\widehat X_t(k)e^{\lambda_k t}$. From the above identity, we obtain the Schr\"odinger equation
\begin{equation}\label{eq:schor}
dY_t(k)= i \theta_k Y_t(k) d {\mcb B}_t,
\end{equation}
whose solutions are of the form
$$
Y_t(k)=Y_0(k)+\int_0^t i \theta_k Y_s(k) d{\mcb B}_t.
$$
Let us assume that there exist two strong solutions $Y_t^1(k)$ and $Y_t^2(k)$ of the above equation starting from the same initial condition. Define $\Phi_t(k):=Y_t^1(k)-Y_t^2(k)$ and $\varphi_t(k)=\bb E\left[\|\Phi_t(k)\|^2\right]$. Thus,
\begin{flalign*}
\varphi_t(k)=\bb E\| \Phi_t(k)\|^2 &=\bb E\left[ \Big\| \int_0^t  i \theta_k \Phi_s(k) d{\mathcal B}_s \Big\|^2\right]= \theta_k^2 \int_0^t \left\| \Phi_s(k)\right\|^2 ds=  \theta_k^2 \int_0^t  \varphi_s(k) ds.
\end{flalign*}
Note that in the third identity above, we used the It\^o's isometry. 
By Gr\"onwall's Lemma and since $\varphi_0(k)=Y_0^1(k)-Y_0^2(k)=0$, we conclude that $\varphi_t(k)=0$ for any $t \geq 0$ and any $k \in \bb Z$. Therefore, the strong solution of the Schr\"odinger equation \eqref{eq:schor} is unique and thus the strong solution of \eqref{arica} is unique.

Now we prove that every solution of the martingale problem given in Definition \ref{def:mart_problem} is a strong solution. 
To that end,  let $\{X_t\}_{t \geq 0}$ be a solution of that martingale problem.  Observe that $\{\mcb M_t(f)\}_{ t \geq 0}$ is a continuous martingale, so, in particular, the representation theorem and the It\^o's formula hold for these martingales. Fix a countable  subset $\mcb D$ of $C^\infty(\bb T)$. 
For each $f \in \mcb D$, let $\{B_t(f)\}_{ t \geq 0}$ be a Brownian motion such that
\[
\mcb M_t(f)  = \int_0^t \sigma X_s( \nabla f) dB_s(f)
\]
for every $t \geq 0$. By the representation theorem, the Brownian motion $\{B_t(f)\}_{ t \geq 0}$ exists in an extended probability space. The importance of  $\mcb D$ to be countable is to guarantee that  all the Brownian motions $\{B_t(f)\}_{ t \geq 0, f \in \mcb D}$ are defined on the same extended space. Our aim is to prove that we can choose $\{B_t(f)\}_{ t \geq 0}$  independently of $f$. For $f,g \in \mcb D$, from the polarization identity and  \eqref{QVofmp},  the joint quadratic variation of $\{\mcb M_t(f)\}_{ t \geq 0}$ and $\{\mcb M_t(g)\}_{ t \geq 0}$ can be computed as follows
\begin{equation}
\label{coquimbo}\begin{split}
\langle \mcb M(f), \mcb M(g)\rangle_t =& \frac{1}{2}\Big(\langle \mcb M(f)+\mcb M(g)\rangle_t-\langle \mcb M(f)\rangle_t-\langle \mcb M(g)\rangle_t \Big)\\
=&\int_0^t \sigma^2\Big(X_s(\nabla f )+X_s (\nabla g)\Big)^2 ds- \int_0^t \sigma^2X_s(\nabla f )^2 ds-\int_0^t \sigma^2X_s (\nabla g)^2 ds\\=&
\int_0^t \sigma^2X_s( \nabla f) \,X_s(\nabla g) ds.\end{split}
\end{equation}
On the other hand,  from the representation theorem (see Theorem 3.4.2 of \cite{KS}), the joint quadratic variation is given by 
\begin{equation}
\label{laserena}
\langle \mcb M(f), \mcb M(g)\rangle_t = \int_0^t \sigma^2 X_s(\nabla f) X_s(\nabla g) \,d \langle B(f), B(g)\rangle_s.
\end{equation}
Let us assume for a moment that $X_t(\sigma \nabla f) \,X_t(\sigma \nabla g) \neq 0$ for every $t \geq 0$. If this is the case, then \eqref{coquimbo} and \eqref{laserena} imply that
\[
\< B(f), B(g)\>_t = t.
\]
Another application of the polarization identity gives that $\<B(f) -B(g)\>_t=0$, from where we conclude that $B_t(f) = B_t(g)$ for every $t \geq 0$. However, we do not have any \emph{a priori} reason to assume that $X_t(\sigma \nabla f) \,X_t(\sigma \nabla  g) \neq 0$ for every $t \geq 0$, nor we can derive this fact from the martingale property. Therefore, we will proceed more carefully. Let $\{f_i\}_{ i \in \bb N}$ be an enumeration of $\mcb D$ and let $\{A_i\}_{ i \in \bb N}$ be given on $i\in\mathbb N$ by
\[
A_i := \{t \geq 0; \, X_t(\sigma \nabla f_i) =0\}.
\]
The representation theorem states that $\{B_t(f)\}_{ t \geq 0}$ admits the representation
\[
B_t(f_i) = \int_0^t \Big( \mathbb{1}_{\{s \notin A_i\}} X_s(\sigma \nabla f_i)^{-1} d \mcb M_s(f_i) + \mathbb{1}_{\{s \in A_i\}}  d {B}^i_s \Big),
\]
where $\{B_t^i\}_{ t \geq0}$ is an auxiliary Brownian motion. According to the representation theorem, we have the freedom to define this auxiliary Brownian motion using a Brownian motion defined in our original filtered space, using a Brownian motion independent of our original filtered space, or using a combination of both procedures. We adopt the latter option, as we explain below. 

Let $A_0:= [0,\infty)$, and define the sets $\{C_i\}_{ i \in \bb N}$ as
\[
C_i:= \bigcap_{j=0}^{i-1}A_j \setminus A_i,
\]
and let $C_\infty : = \cap_{i \in \bb N} A_i$. Observe that $C_1=A_0\setminus A_1$, $C_2=(A_0\cap A_1)\setminus A_2$ and so on.   Note that $\{C_i\}_{ i \in \bb N\cup \infty}$ is a partition of $A_0$.
Let $\{W_s\}_{s\geq 0}$ be a Brownian motion independent of $\{B_s^i\}_{s\geq 0}$ and $\{\mcb M_s(f_i)\}_{s\geq0}$, for all $i\geq 1$. Define 
\[
\widetilde{B}_t(f_1) := \int_0^t \Big( \sum_{i=1}^\infty \mathbb{1}_{\{s \in C_i\}} X_s(\sigma \nabla f_i)^{-1} d \mcb M_s(f_i) + \mathbb{1}_{\{s \in C_\infty\}} d {W_s} \Big).
\]
From the representation theorem, we have that
\[
\mcb M_t(f_1) = \int_0^t X_s(\sigma \nabla f_1) d \widetilde{B}_s(f_1).
\]
Recall that by \eqref{coquimbo} and \eqref{laserena}, $d B_t(f_1) = d B_t(f_2)$ for $t \notin A_1 \cup A_2$. We also have that 
\[
d \widetilde{B}_t(f_1) = d {B}_t(f_1) \text{ for } t \notin A_1 \cup A_2,
\]
\[
d \widetilde{B}_t(f_1) = d B_t(f_2) \text{ for } t \in A_1 \setminus A_2,
\]
from where we conclude that $d \widetilde{B}_t(f_1) = d B_t(f_2)$ for $t \notin A_2$. Let us define 
\[
\widetilde{B}_t(f_2) := \int_0^t \Big( \mathbb{1}_{\{ s \notin A_2\}} d B_s(f_2) + \mathbb{1}_{\{s \in A_2\}} d \widetilde{B}_s(f_1)\Big)
\]
and observe that
\[
\mcb M_t(f_2) = \int_0^t X_s(\sigma \nabla f_2) d \widetilde{B}_s(f_2)
\]
and $\widetilde{B}_t(f_1) = \widetilde{B}_t(f_2)$ for every $t \geq 0$.
Inductively, we can define
\[
\widetilde{B_t}(f_{\ell+1}) := \int_0^t \Big( \mathbb{1}_{\{s \notin A_{\ell+1}\}} d B_s(f_{\ell+1}) + \mathbb{1}_{\{s \in A_{\ell+1}\}} d \widetilde{B}_s(f_\ell)\Big)
\]
and we can check that $\widetilde{B}_t(f_{\ell+1}) = \widetilde{B}_t(f_\ell)$ for every $t \geq 0$. We conclude that $\widetilde{B}_t:= \widetilde{B}_t(f_\ell)$ does not depend on $\ell$ and 
\[
\mcb M_t(f_\ell) = \int_0^t X_s(\sigma \nabla f_\ell) d \widetilde{B}_s(f_\ell)
\]
for every $\ell \in \bb N$. This shows that the solution of the martingale problem $\{X_t\}_{ t \geq 0}$ of \eqref{arica} satisfies
\[
X_t(f) = X_0(f) + \int_0^t X_s(\nu \Delta f - \gamma \nabla f) ds + \int_0^t X_s(\sigma \nabla f) d \widetilde{B}_s
\]
for every $f \in \mcb D$. Observe that in order to show uniqueness of strong solutions of \eqref{arica}, we only need to verify \eqref{iquique} for the trigonometric functions $\{g_k\}_{ k \in \bb Z}$.
Choosing $\mc D =: \{ \textrm{Re} (g_k), \textrm{Im} (g_k)\}_{ k \in \bb Z}$, we conclude that every martingale solution of \eqref{arica} is also a strong solution of \eqref{arica}, from where uniqueness follows.

\begin{center}
\textbf{Acknowledgments}
\end{center}
P.G. thanks  FCT/Portugal for financial support
through CAMGSD, IST-ID, projects UIDB/04459/2020 and UIDP/04459/2020.  R.M. thanks the National Council for Scientific and Technological Development (CNPq). M.J.~has been funded by  CNPq grant 201384/2020-5 and FAPERJ grant E-26/201.031/2022. 
This work was partially supported by the National Council for Scientific and Technological Development (CNPq),  by the Simons Foundation and by the Mathematisches Forschungsinstitut Oberwolfach.
This project has received funding from the European Research Council (ERC) under the European Union's Horizon $2020$ research and innovative programme (grant agreement No $715734$).

\end{document}